\documentclass[12pt]{amsart}
\usepackage{amsmath}
\usepackage{amsxtra}
\usepackage{amstext}
\usepackage{amssymb}
\usepackage{amsthm}
\usepackage{latexsym}
\usepackage{dsfont} 
\usepackage{verbatim}
\usepackage{tabls}
\usepackage{graphicx}
\usepackage{rotating}
\usepackage{caption}
\usepackage[shortlabels]{enumitem}

\theoremstyle{plain}
\newtheorem{thm}{Theorem}[section]
\newtheorem{thmx}{Theorem}

\newtheorem{cor}[thm]{Corollary}
\newtheorem{lem}[thm]{Lemma}
\newtheorem{prop}[thm]{Proposition}

\newtheorem*{principle}{General Principle}
\newtheorem{alt}[thm]{Alternative}

\theoremstyle{definition}
\newtheorem{defn}[thm]{Definition}

\numberwithin{equation}{section}
\def\R1{\widetilde{R}}
\def\T1{\widetilde{T}}
\def\dist{\operatorname{dist}}
\def\supp{\operatorname{supp}}
\def\Lip{\operatorname{Lip}}
\def\eps{\varepsilon}
\def\kap{\varkappa}
\def\R{\mathbb{R}}

\def\dy{\mathcal{D}}

\def\dyup{\mathcal{D}_{\operatorname{up}}}
\def\dyupk{\mathcal{D}^{(k)}_{\;\;\;\;\operatorname{up}}}

\def\gdown{\mathcal{G}_{\operatorname{down}}}

\def\wh{\widehat}
\def\wt{\widetilde}

\def\Span{\operatorname{span}}

\def\lis{\lfloor s\rfloor}
\def\uis{\lceil s\rceil}
\def\Ups{\Upsilon}
\def\wlim{\rightharpoonup}
\def\S{\mathcal{S}}
\def\I{\mathcal{I}}
\def\kap{\varkappa}

\def\Xint#1{\mathchoice
   {\XXint\displaystyle\textstyle{#1}}%
   {\XXint\textstyle\scriptstyle{#1}}%
   {\XXint\scriptstyle\scriptscriptstyle{#1}}%
   {\XXint\scriptscriptstyle\scriptscriptstyle{#1}}%
   \!\int}
\def\XXint#1#2#3{{\setbox0=\hbox{$#1{#2#3}{\int}$}
     \vcenter{\hbox{$#2#3$}}\kern-.5\wd0}}

\def\dashint{\Xint-}

\begin{document}

\title[The boundedness of Square function operators]
{The measures with an associated square function operator bounded in $L^2$}

\author[B. Jaye]{Benjamin Jaye}
\address{Department of Mathematical Sciences, Kent State University, Kent, Ohio 44240, USA}
\email{bjaye@kent.edu}
\author[F. Nazarov]{Fedor Nazarov}
\address{Department of Mathematical Sciences, Kent State University, Kent, Ohio 44240, USA}
\email{nazarov@math.kent.edu}
\author[X. Tolsa]{Xavier Tolsa}
\address{ICREA, Passeig Llu\'\i ­s Companys 23 08010 Barcelona, Catalonia, and Departament de Matem\`atiques and BGSMath, Universitat Autònoma de Barcelona, 08193 Bellaterra (Barcelona), Catalonia.}
\email{xtolsa@mat.uab.cat}


\thanks{B.J. was supported in part by NSF DMS-1500881. F.N. was supported in part by NSF DMS-1600239.  X.T. was supported by the ERC grant 320501 of the European Research Council (FP7/2007-2013), and also partially supported by 2014-SGR-75 (Catalonia), MTM2013-44304-P (Spain), and by the Marie Curie ITN MAnET (FP7-607647).}

\date{\today}
\maketitle
\tableofcontents

\section{Introduction}

 Fix $d\geq 2$ and $s\in (0,d)$.  The aim of this paper is to provide an extension of a theorem of David and Semmes \cite{DS} to general non-atomic measures.  Their theorem provides a geometric characterization of the $s$-dimensional Ahflors-David regular measures\footnote{A measure $\mu$ is Ahflors-David regular if there exists a constant $C>0$ such that $\frac{1}{C}r^s\leq \mu(B(x,r))\leq Cr^s$ for every $x\in \supp(\mu)$ and $r>0$.} for which a certain class of square function operators, or singular integral operators, are bounded in $L^2(\mu)$.\\

 Their description is given in terms of Jones' $\beta$-coefficients, which are defined for $s\in \mathbb{N}$ as
 $$\beta_{\mu}(B(x,r)) = \Bigl(\frac{1}{\mu(B(x,r))}\inf_{L\in \mathcal{P}_s}\int_{B(x,r)}\Bigl(\frac{\dist(y,L)}{r}\Bigl)^2d\mu(y)\Bigl)^{1/2},
 $$
where $B(x,r)$ denotes the open ball centred at $x\in \R^d$ with radius $r>0$, and $\mathcal{P}_s$ denotes the collection of affine $s$-planes in $\R^d$.   Jones introduced these coefficients (with $L^{\infty}(\mu)$ norm replacing the $L^2(\mu)$ mean) in order to give a new proof of the boundedness of the Cauchy Transform on a Lipschitz curve \cite{Jo1} and to characterize the rectifiable curves in $\R^2$ \cite{Jo2}. \\

 Let us now state the David-Semmes theorem in the form most convenient for our purposes.


\begin{thmx}\label{DSthm}\cite{DS}  Suppose that $\mu$ is an $s$-dimensional Ahlfors-David regular measure. The following three statements are equivalent:

(i)  for every odd function $K\in C^{\infty}(\R^d\backslash \{0\})$ satisfying standard decay estimates\footnote{Namely, that for every multi-index $\alpha$,  there is a constant $C_{\alpha}>0$ such that $|D^{\alpha}K(x)|\leq \frac{C_{\alpha}}{|x|^{s+|\alpha|}}$ for every $x\in \R^d\backslash \{0\}$.}, and $\eps>0$, the truncated singular integral operator (SIO)
\begin{equation}\label{DSSIO}
T_{\mu,\eps}(f)(\,\cdot\,) =  \int_{\R^d\backslash B(x,\eps)}K(\, \cdot\,-y)f(y)d\mu(y)
\end{equation}
is bounded on $L^2(\mu)$ with an operator norm that can be estimated independently of $\eps$.

(ii)  for every odd function $\psi\in C^{\infty}_0(\R^d)$, the square function operator
\begin{equation}\label{DSsquare}S_{\mu, \psi}(f)(\,\cdot\,)=\Big[\int_0^{\infty}\Bigl|\frac{1}{t^s}\int_{\R^d}\psi\Bigl(\frac{\,\cdot\,-y}{t}\Bigl)f(y)d\mu(y)\Bigl|^2\frac{dt}{t}\Big]^{\tfrac{1}{2}}
\end{equation}
is bounded in $L^2(\mu)$.

(iii)  $s\in \mathbb{Z}$ and there exists a constant $C>0$ such that
\begin{equation}\label{geosquare}\int_Q\int_0^{\ell(Q)}\beta_{\mu|Q}(B(x,r))^2\frac{dr}{r}d\mu(x)\leq C\mu(Q)
\end{equation}
 for every cube $Q\subset \R^d$, where $\mu|Q$ denotes the restriction of $\mu$ to $Q$.
\end{thmx}

We shall henceforth refer to (i) as the condition that \emph{all SIOs with smooth odd kernels are bounded in $L^2(\mu)$}.


The path that David and Semmes take to prove Theorem \ref{DSthm} is to show that condition (ii) implies (iii), and also that (iii) is equivalent to a number of geometric conditions on the support of $\mu$, such as \emph{uniform rectifiability} (see \cite{DS} for definitions).  One can then apply a theorem of David \cite{Dav} to conclude that (i) holds.   A standard artifice takes us from (i) to (ii) (see Section \ref{sintsquare} below).\\

At this point we should mention that David and Semmes asked whether replacing the condition (i) with just the $L^2(\mu)$ boundedness of the $s$-Riesz transform -- the SIO with kernel $K(x) = \tfrac{x}{|x|^{s+1}}$ -- is already sufficient to conclude that (iii) holds.  The fact that $s\in \mathbb{Z}$ under this assumption was proved by Vihtil\"{a} \cite{Vih}.  Demonstrating that (\ref{geosquare}) holds if $s\in \mathbb{Z}$ has proven more elusive, and is at present only known when $s=1$, by the Mattila-Melnikov-Verdera theorem \cite{MMV}, and $s=d-1$, when it was proved by Nazarov-Tolsa-Volberg \cite{NToV} (in an equivalent form).\\

In this paper, we do not make any progress on the Riesz transform question, but instead give a complete solution to another problem of David and Semmes referred to (rather generously) in Section 21 of \cite{DS} as a ``glaring omission'' in their theorem.  Namely, we provide an analogue of Theorem~\ref{DSthm} for general non-atomic locally finite Borel measures (without any regularity assumptions).  Moreover, we do so for the somewhat smaller class of singular integral kernels considered by Mattila and Preiss \cite{MP}.  When specialized to the case of Ahflors-David regular measures, our arguments yield a new direct proof of the assertion that (ii) implies (iii) in Theorem \ref{DSthm} above.


\subsection{The non-integer condition: The Wolff Energy}  The conditions that should replace (iii) in Theorem \ref{DSthm} when one considers a general measure are by now quite well agreed upon by specialists.  This is particularly true when $s\not\in \mathbb{Z}$, due to the work of Mateu-Prat-Verdera \cite{MPV}.  It turned out that a well-known object in non-linear potential theory, \emph{the Wolff energy}, provides the key.  We define the Wolff energy of a cube $Q\subset \R^d$ by
$$\mathcal{W}(\mu, Q)=\int_{Q}\int_0^{\infty}\Bigl(\frac{\mu(Q\cap B(x,r))}{r^s}\Bigl)^2\frac{dr}{r}d\mu(x).
$$
The \emph{Mateu-Prat-Verdera theorem} states that, if $s\in (0,1)$, then for a non-atomic measure $\mu$, the $s$-Riesz transform of $\mu$ is bounded in $L^2(\mu)$ if and only if the following Wolff energy condition holds:
\begin{equation}\label{Wolffenergy}
\mathcal{W}(\mu, Q)\leq C\mu(Q)\text{ for every cube }Q\subset \R^d.
\end{equation}

In the proof presented in \cite{MPV}, the necessity of the Wolff energy condition for the boundedness of the $s$-Riesz transform relied fundamentally on the restriction to $s\in (0,1)$, as it made use of a variation of the Menger-Melnikov curvature formula.  However, the sufficiency of the condition (\ref{Wolffenergy}) relied on neither the particular structure of the $s$-Riesz kernel $\tfrac{x}{|x|^{s+1}}$, nor the restriction on $s$, and by adapting their technique one can prove the following result.

\begin{thmx}[Mateu-Prat-Verdera]\label{suffwolff}  Fix $s\in (0,d)$.  If $\mu$ is a measure that satisfies (\ref{Wolffenergy}), then all $s$-dimensional SIOs with all smooth odd kernels are bounded in $L^2(\mu)$ (that is, statement (i) of Theorem \ref{DSthm} holds).
\end{thmx}

To find a proof of this theorem precisely as stated, one can consult Appendix A of \cite{JN2}.  The \emph{Mateu-Prat-Verdera conjecture} asks whether one may extend the necessity of the condition (\ref{Wolffenergy}) for the $L^2(\mu)$ boundedness of the $s$-Riesz transform in $L^2(\mu)$ to the range $s>1$, $s\not\in \mathbb{Z}$.  This was recently proved in the case when $s\in (d-1,d)$ by M.-C. Reguera and the three of us \cite{JNRT}.  It is an open problem for $s\in (1,d-1)\backslash\mathbb{Z}$.

\subsection{The integer condition:  The Jones Energy}  For the case of  integer $s$, we introduce the Jones energy of a cube $Q\subset \R^d$:
\begin{equation}\label{JWenergy}\mathcal{J}(\mu, Q) = \int_Q \int_0^{\infty} \Bigl[\beta_{\mu|_Q}(B(x,r))^2 \Bigl(\frac{\mu(Q\cap B(x,r))}{r^s}\Bigl)^2\Bigl]\frac{dr}{r}d\mu(x).
\end{equation}
Here $\mu|Q$ denotes the restriction of $\mu$ to $Q$.  This square function appears in Azzam-Tolsa \cite{AT}, where amongst other things, the following theorem is proved.

\begin{thmx}\cite{AT} Let $\mu$ be a non-atomic measure on $\mathbb{C}$.  Then the Cauchy transform, the one dimensional SIO with kernel $K(z) = \frac{1}{z}$ in $\mathbb{C}$, is bounded in $L^2(\mu)$, if and only if $\sup_{z\in \mathbb{C}, r>0}\frac{\mu(B(z,r))}{r^s}\leq C$ and
$$\mathcal{J}(\mu,Q)\leq C\mu(Q) \text{ for every cube }Q\subset \mathbb{C}.
$$
\end{thmx}

This theorem makes essential use of the relationship between the $L^2$-norm of the Cauchy transform of a measure, and the curvature of a measure.  Nevertheless, by combining the techniques of \cite{AT} with those in \cite{Tol2}, Girela-Sarri\'{o}n \cite{G} succeeded in proving the sufficiency of the Jones energy condition for the boundedness of SIOs in greater generality:

\begin{thmx}\label{suffjones}\cite{G} Fix $s\in \mathbb{Z}$, $s\in (0,d)$.  Suppose that there is a constant $C>0$ such that $\sup_{x\in \R^d}\frac{\mu(B(x,r))}{r^s}\leq C$ and
 \begin{equation}\label{JWCarleson}\mathcal{J}(\mu,Q)\leq C\mu(Q) \text{ for every cube }Q\subset \R^d.
\end{equation}
Then all $s$-dimensional SIOs with smooth odd kernels are bounded in $L^2(\mu)$.
\end{thmx}

\subsection{Statement of results} Choose a non-negative non-increasing function $\varphi\in C^{\infty}([0,\infty))$, such that $\supp(\varphi)\subset [0,2)$ and $\varphi\equiv 1$ on $[0, 1)$.  We form the square function operator
$$\mathcal{S}_{\mu}(f)(x) = \Bigl(\int_0^{\infty}\Bigl|\int_{\R^d}\frac{x-y}{t^{s+1}}\varphi\Bigl(\frac{|x-y|}{t}\Bigl)f(y)d\mu(y)\Bigl|^2\frac{dt}{t}\Bigl)^{1/2}.
$$


We shall prove the following two results:

\begin{thm}\label{thm2}  Fix $s\not\in \mathbb{Z}$.  Let $\mu$ be a non-atomic locally finite Borel measure.   If the square function operator $\S_{\mu}$ is bounded in $L^2(\mu)$, then there is a constant $C>0$ such that
\begin{equation}\label{wolffenergy}\mathcal{W}(\mu, Q)\leq C\mu(Q)
\end{equation}
for every cube $Q\subset \R^d$.
\end{thm}



\begin{thm}\label{thm1}  Fix $s\in \mathbb{Z}$.  Let $\mu$ be a non-atomic locally finite Borel measure.   If the square function operator $\S_{\mu}$ is bounded in $L^2(\mu)$, then there is a constant $C>0$ such that $\frac{\mu(B(x,r))}{r^s}\leq C$ for every $x\in \R^d,\, r>0$, and
\begin{equation}\label{jonesenergy}\mathcal{J}(\mu, Q)\leq C\mu(Q)
\end{equation}
for every cube $Q\subset \R^d$.
\end{thm}

\subsection{Singular integrals and square functions}\label{sintsquare} When combined with the theorems of Mateu-Prat-Verdera \cite{MMV} and Girela-Sarri\'{o}n \cite{G} (Theorems \ref{suffwolff} and \ref{suffjones} above), our theorems yield the following result.

\begin{thm}\label{equivalences}  Suppose that $\mu$ is a non-atomic locally finite Borel measure.  The following statements are equivalent.
\begin{enumerate}[(i)]
\item All SIOs with smooth odd kernels are bounded in $L^2(\mu)$.
\item All SIOs of Mattila-Preiss type are bounded in $L^2(\mu)$.  These are the SIOs with kernels that have the form $K(x) = \frac{x}{|x|^{s+1}}\psi(|x|)$  for $\psi\in C^{\infty}([0,\infty))$ satisfying $$|\psi^{(k)}(t)|\leq C_k|t|^{-k} \text{ for every }t\in [0,\infty) \text{ and every }k\geq 0.$$
\item  The square function operator $\S_{\mu}$ is bounded in $L^2(\mu)$.
\item Either
\begin{itemize}
\item  $s\not\in \mathbb{Z}$ and the Wolff energy condition (\ref{wolffenergy}) holds,
\end{itemize}
or
\begin{itemize}
\item $s\in \mathbb{Z}$ and the Jones energy condition (\ref{jonesenergy}) holds.
\end{itemize}
\end{enumerate}
\end{thm}

 That (iii) implies (iv), is merely a restatement of Theorems \ref{thm2} and \ref{thm1}, while Theorems \ref{suffwolff} and \ref{suffjones} imply that (iv) implies (i).  That (i) implies (ii) is trivial as every SIO of Mattila-Preiss type is a SIO with smooth odd kernel.  Thus we only need to show that (ii) implies (iii).  This is a standard argument, already present in \cite{DS, MP}.  To sketch the idea, let us fix a sequence $\eps_k$ of independent mean zero $\pm 1$-valued random variables (on some probability space $\Omega$).  For $\omega\in \Omega$, $t\in [1,2)$, and $k_0\in \mathbb{N}$, consider the following SIO of Mattila-Preiss type
$$T_{t,k_0,\omega}(f)(x) = \int_{\R^d}\Bigl[\sum_{k\in \mathbb{Z}, |k|\leq k_0}\eps_k(\omega)\frac{x-y}{(2^kt)^{(s+1)}}\varphi\Bigl(\frac{|x-y|}{2^kt}\Bigl)\Bigl]f(y)d\mu(y).
$$
Following Section 3 of \cite{DS}, one obtains that $$ \|S_{\mu}(f)\|_{L^2(\mu)}^2\leq C \sup_{k_0\in \mathbb{N}}\int_{1}^2\mathbb{E}_{\omega}\|T_{t, k_0,\omega}(f)\|^2_{L^2(\mu)}\frac{dt}{t}\leq C\|f\|_{L^2(\mu}^2,$$
since all SIOs of Mattila-Preiss type are bounded in $L^2(\mu)$.

Proving that (iii) implies (ii) or (ii) implies (i) without going through (iv) appears to be non-trivial. (At least we do not know how to do that.)

\subsection{The particular choice of the bump function $\varphi$ doesn't matter too much.}  It is natural to wonder the extent to which the mapping properties of $\S_{\mu}$ depend on the particular choice of the bump function $\varphi$.  Here we make three remarks in this regard, with the particular aim of convincing the reader that Theorems \ref{thm2} and \ref{thm1} remain valid if one instead defines the square function operator in a more customary way with a (perhaps only bounded measurable) bump function that is supported away from $0$.

(1)  Suppose that $\psi\in C^{\infty}([0,\infty))$ is a non-negative function that has bounded support and is identically equal to $1$ near $0$.  Then the proofs of Theorems \ref{thm2} and \ref{thm1} can be adapted so that the same conclusions are reached with the $L^2(\mu)$ boundedness of $\S_{\mu}$ replaced by that of the operator
$$\mathcal{S}_{\mu,\psi}(f)(x) = \Bigl(\int_0^{\infty}\Bigl|\int_{\R^d}\frac{x-y}{t^{s+1}}\psi\Bigl(\frac{|x-y|}{t}\Bigl)f(y)d\mu(y)\Bigl|^2\frac{dt}{t}\Bigl)^{1/2}.
$$

(2) For non-negative functions $\psi$ and $g$, define the multiplicative convolution $$\psi_g(t) = \int_{0}^{\infty}\psi\Bigl(\frac{t}{u}\Bigl)g(u)\frac{du}{u}.$$   From a change of variable and Minkowski's inequality we infer that
$$\|\S_{\mu, \psi_g}(f)\|_{L^2(\mu)}\leq \Bigl[\int_{0}^{\infty} u^{s+1} g(u) \frac{du}{u}\Bigl] \|\S_{\mu,\psi}(f)\|_{L^2(\mu)},
$$
and as such, if $\S_{\mu, \psi}$ is bounded in $L^2(\mu)$, and $\int_{0}^{\infty} u^{s}g(u)du<\infty$, then $\S_{\mu, \psi_g}$ is bounded in $L^2(\mu)$.

(3) Finally, suppose that $\psi$ is non-negative, bounded, measurable, and compactly supported in $(0,\infty)$ (so $0\not\in \supp(\psi)$), with $\S_{\mu, \psi}$ bounded on $L^2(\mu)$.

Writing $\supp(\psi)\subset [a,A]$ for some $a,A>0$, we choose a function $g\in C^{\infty}([0,\infty))$ supported on $[0,\tfrac{2}{a}]$ that takes the value $\bigl(\int_0^{\infty}\psi\bigl(\tfrac{1}{u}\bigl)\tfrac{du}{u}\bigl)^{-1}$ on the interval $[0,\tfrac{1}{a}]$.  Then the function $\psi_g\in C^{\infty}([0,\infty))$ has support contained in $\bigl[0, \tfrac{2A}{a}\bigl]$ and $\psi_g\equiv 1$ on $[0,1]$. From remark (2) we have that $\S_{\mu, \psi_g}$ is bounded on $L^2(\mu)$.

\subsection{The Mayboroda-Volberg Theorem} Building on the tools developed in \cite{Tol2, RdVT}, Mayboroda and Volberg \cite{MV1, MV2} proved that if $\mu$ is a non-trivial finite measure with $\mathcal{H}^s(\supp(\mu))<\infty$, and $\S_{\mu}(1)<\infty$ $\mu$-almost everywhere, then $s\in \mathbb{Z}$ and $\supp(\mu)$ is $s$-rectifiable (see Section \ref{basicdensity} below for the definition).  When combined with Theorem 1.1 of Azzam-Tolsa \cite{AT}, Theorems \ref{thm2} and \ref{thm1} above provide another demonstration of this result.  We sketch the argument here.

One begins with a standard $T(1)$-theorem argument which involves finding a compact subset $E\subset \supp(\mu)$ whose $\mu$ measure is as close to $\mu(\R^d)$ as we wish, for which $\S_{\mu'}$ is bounded in $L^2(\mu')$ with $\mu' = \mu|E$.  This utilizes the method of suppressed kernels, see for instance Proposition 3.2 of \cite{MV1}.  But since $\mu'$ is supported on a set of finite $\mathcal{H}^s$ measure, the conclusion of Theorem \ref{thm2} cannot hold unless $\mu'\equiv 0$, and so $s\in \mathbb{Z}$ and the conclusion of Theorem \ref{thm1} holds.  Theorem 1.1 in \cite{AT} then yields that $\supp(\mu')$ is rectifiable.  From this we conclude that $\supp(\mu)$ is rectifiable.

\section{Preliminaries}

\subsection{Notation}
\begin{itemize}
\item By $C>0$  we denote a constant that may change from line to line.  Any constant may depend on $d$ and $s$ without mention.  If a constant depends on parameters other than $d$ and $s$, then these parameters are indicated in parentheses after the constant.
\item We denote the closure of a set $E$ by $\overline{E}$.
\item For $x\in \R^d$ and $r>0$, $B(x,r)$ denotes the open ball centred at $x$ with radius $r$.
\item By a measure, we shall always mean a non-negative locally finite Borel measure.
\item  We denote by $\Lip(\R^d)$ the collection of Lipschitz continuous functions on $\R^d$.  For an open set $U$, we denote by $\Lip_0(U)$ the subset of $\Lip(\R^d)$ consisting of those Lipschitz continuous functions with compact support in $U$.  We define the homogeneous Lipschitz semi-norm
    $$\|f\|_{\Lip} = \sup_{x,y\in \R^d, x\neq y}\frac{|f(x)-f(y)|}{|x-y|}.
    $$
\item We denote by $\supp(\mu)$ the closed support of $\mu$, that is, $$\supp(\mu) = \mathbb{R}^d\backslash \bigl\{\cup B : B \text{ is an open ball with }\mu(B)=0\bigl\}.$$
\item For a closed set $E$, we shall denote by $\mu|_E$ the restriction of the measure $\mu$ to $E$, that is, $\mu|_E(A) = \mu(A\cap E)$ for a Borel set $A$.
\item For $n\geq 0$, we denote by $\mathcal{H}^{n}$ the $n$-dimensional Hausdorff measure.  When restricted to an $n$-plane, $\mathcal{H}^n$ is equal to a constant multiple of the $n$-dimensional Lebesgue measure $m_n$.
\item For a cube $Q\subset \R^d$, $\ell(Q)$ denotes its side-length. For $A>0$, we denote by $AQ$ the cube concentric to $Q$ of side-length $A\ell(Q)$.
\item Set $Q_0= (-\tfrac{1}{2},\tfrac{1}{2})^d$.  For a cube $Q$, we set $\mathcal{L}_Q$ to be the canonical affine map (a composition of a dilation and a translation) satisfying $\mathcal{L}_Q(Q_0)=Q$.
\item We define the ratio of two cubes $Q$ and $Q'$ by
$$[Q':Q]=\Bigl|\log_2 \frac{\ell(Q')}{\ell(Q)} \Bigl|.$$
\item For any $x\in \R^d$, $r>0$, we set
$$\I_{\mu}(B(x,r)) = \int_{\R^d}\varphi\Bigl(\frac{|x-y|}{r}\Bigl)d\mu(y),
$$
so $\mu(B(x,r))\leq \I_{\mu}(B(x,r))\leq \mu(B(x,2r))$.
\end{itemize}




\subsection{Balls associated to cubes}

 We associate the ball $B_{Q_0} = B(0, 4\sqrt{d})$ to the cube $Q_0 = (-\tfrac{1}{2},\tfrac{1}{2})^d$.  Then for an arbitrary cube $Q$, we set
$$B_Q = \mathcal{L}_Q(B_{Q_0}).
$$
Notice that $B_Q = B(x_Q, 4\sqrt{d}\ell(Q))$, where $x_Q = \mathcal{L}_Q(0)$ is the centre of $Q$.

We associate to the cube $Q_0$ the function $\varphi_{Q_0}(x) = \varphi(\tfrac{|x|}{2\sqrt{d}})$, $x\in \R^d$.  For any other cube $Q$ we set $\varphi_{Q}= \varphi_{Q_0}\circ \mathcal{L}_Q^{-1} = \varphi\bigl(\frac{|\,\cdot\,-x_Q|}{2\sqrt{d} \ell(Q)}\bigl)$.  The reader may wish to keep in mind the following chain of inclusions:
$$3Q\subset B(x_Q, 2\sqrt{d}\ell(Q))\subset \{\varphi_Q= 1\}\subset \supp(\varphi_Q)\subset B_Q.
$$
We set
$$\I_{\mu}(Q) = \int_{\R^d}\varphi_Q\,d\mu \;\;\Bigl(= \int_{B_Q}\varphi_Q\,d\mu\Bigl).
$$
In relation to our previous notation, we have $\I_{\mu}(Q) = \I_{\mu}(\tfrac{1}{2}B_Q)$.
For $n>0$, we define the $n$-density of a cube $Q$ by
$$D_{\mu,n}(Q) = \frac{1}{\ell(Q)^n}\int_{\R^d}\varphi_Q d\mu =  \frac{1}{\ell(Q)^n}\mathcal{I}_{\mu}(Q).
$$
Thus
\begin{equation}\label{densequiv}\frac{\mu(Q)}{\ell(Q)^n}\leq D_{\mu,n}(Q)\leq \frac{\mu(B_Q)}{\ell(Q)^n}\leq \frac{\mu(8\sqrt{d}Q)}{\ell(Q)^n}.
\end{equation}
If $n=s$, then we just write $D_{\mu}(Q)$ instead of $D_{\mu,s}(Q)$.

\subsection{Flatness and transportation coefficients}
For $n\in \mathbb{N}$, the $n$-dimensional $\beta$-coefficient of a measure $\mu$ in a cube $Q$ is given by
$$\beta_{\mu,n}(Q) = \Bigl[\frac{1}{\I_{\mu}(Q)}\inf_{L\in \mathcal{P}_n}\int_{\R^d}  \Bigl(\frac{\dist(x,L)}{\ell(Q)}\Bigl)^2 \varphi_Q(x) d\mu(x)\Bigl]^{1/2},
$$
where, as before, $\mathcal{P}_n$ denotes the collection of $n$-planes in $\R^d$. We shall write
$$\beta_{\mu}(Q) = \beta_{\mu, \lfloor s\rfloor}(Q).
$$

It is easy to see that there is a $n$-plane $L_Q$ such that
$$\beta_{\mu,n}(Q) = \Bigl[\frac{1}{\I_{\mu}(Q)}\int_{\R^d} \Bigl(\frac{\dist(x,L_Q)}{\ell(Q)}\Bigl)^2  \varphi_Q(x)d\mu(x)\Bigl]^{1/2},
$$
and we shall call any plane $L_Q$ satisfying this property an optimal $n$-plane for $\beta_{\mu,n}(Q)$.  The following classical fact will prove very useful for our analysis:

\begin{lem}\label{leastsquares}  Suppose $\nu$ is a non-zero finite measure.  Every $n$-plane $L$ that minimizes the quantity $\int_{\R^d}\dist(x,L)^2d\nu(x)$ contains the centre of mass of $\nu$, that is, the point $\frac{1}{\nu(\R^d)}\int_{\R^d}x\,d\nu(x)\in \R^d$.
\end{lem}

\begin{proof}  We may assume that $\int_{\R^d}x \,d\nu(x)=0.$  For a $(d-n)$-dimensional orthonormal set $v_{n+1}, \dots, v_d$, consider the function $F:\R^d\to \R$ given by
$$F(b) = \int_{\R^d}\Bigl|\sum_{j=n+1}^d\langle (b-x), v_j\rangle v_j\Bigl|^2d\nu(x), \; b\in \R^d.
$$
For the $n$-plane $L = b+\text{span}(v_{n+1},\dots, v_d)^{\perp}$ to be a minimizer, we must certainly have that $\nabla F(b)=0$.  But
$$\nabla F(b) =\int_{\R^d} 2\Bigl(\sum_{j=n+1}^d\langle (b-x),v_j\rangle v_j \Bigl)d\nu(x) =2\nu(\R^d)\sum_{j=n+1}^{d}\langle b,v_j\rangle v_j.$$

Thus  $\nabla F(b)=0$ if and only if  $b\in \text{span}(v_{n+1},\dots, v_d)^{\perp}$. Therefore, should $L$ be optimal, then it is necessarily a linear subspace.
\end{proof}

The $n$-dimensional transportation (or Wasserstein) coefficient of a measure $\mu$ in a cube $Q\subset \R^d$ is given by
\begin{equation}\nonumber \alpha_{\mu,n}(Q) = \inf_{\substack{L\in \mathcal{P}_n:\\ L\cap \tfrac{1}{4}B_Q\neq\varnothing}}\sup_{\substack{f\in \Lip_0(3B_Q),\\ \|f\|_{\Lip}\leq \tfrac{1}{\ell(Q)}}}\Bigl|\int_{ \R^d}\varphi_Qf\,d(\mu-\vartheta_{\mu, L}\mathcal{H}^n|_L)\Bigl|,
\end{equation}
where $\vartheta_{\mu,L} = \frac{\I_{\mu}(Q)}{\I_{\mathcal{H}^n|_L}(Q)}$.
In the case when $n=s$ we as will write $\alpha_{\mu}(Q) = \alpha_{\mu, s}(Q)$.


Notice that the $\beta$-number is a gauge of how flat the measure is within a given cube, while the $\alpha$-number tells us how close a measure is to a constant multiple of the Lebesgue measure of an $n$-plane.  As one might expect, for $n\in \mathbb{N}$, we have
$$\beta_{\mu,n}(Q)^2 \leq C\alpha_{\mu,n}(Q).
$$

To see this, take an $n$-plane $L$ that intersects $\tfrac{1}{4}B_Q$.  Then the function
$$f(x) = \Bigl(\frac{\dist(x,L)}{\ell(Q)}\Bigl)^2\varphi_{3Q}
$$
is supported in $3B_Q$ and has Lipschitz norm bounded by $\frac{C}{\ell(Q)}$.  This proves the desired inequality since $\varphi_{3Q}\varphi_Q = \varphi_Q$.

\subsection{The dyadic energies}\label{energies} Consider a dyadic lattice $\dy$.  Then, for any finite measure $\mu$ we have the following two inequalities:

\begin{equation}\label{dyjonesbigger}\mathcal{J}(\mu, \R^d)\leq C\sum_{Q\in \dy}\beta_{\mu}(Q)^2D_{\mu}(Q)^2\I_{\mu}(Q),
\end{equation}
and
\begin{equation}\label{dywolffbigger}\mathcal{W}(\mu, \R^d)\leq C\sum_{Q\in \dy}D_{\mu}(Q)^2\I_{\mu}(Q).
\end{equation}

Both of these inequalities follow from integrating with respect to $\mu$ the pointwise inequalities, for $s\in \mathbb{Z}$,
$$\int_0^{\infty}\beta_{\mu}(B(x,r))^2\Bigl(\frac{\mu(B(x,r))}{r^s}\Bigl)^2\frac{dr}{r}\leq C\sum_{Q\in \dy}\beta_{\mu}(Q)^2D_{\mu}(Q)^2\varphi_Q(x),
$$
and, for $s\in (0,d)$,
$$\int_0^{\infty}\Bigl(\frac{\mu(B(x,r))}{r^s}\Bigl)^2\frac{dr}{r}\leq C\sum_{Q\in \dy}D_{\mu}(Q)^2\varphi_Q(x).
$$
We shall just prove the first pointwise inequality (the second one is easier).  Rewrite the left hand side as
\begin{equation}\label{splitcontjones}\sum_{k\in \mathbb{Z}} \int_{2^k}^{2^{k+1}}\Bigl[\frac{\mu(B(x,r))}{r^s}\inf_{L\in \mathcal{P}_s}\frac{1}{r^s}\int_{B(x,r)}\Bigl(\frac{\dist(y,L)}{r}\Bigl)^2d\mu(y)\Bigl]\frac{dr}{r}.
\end{equation}
For each $x\in \R^d$ and $k\in \mathbb{Z}$, there is a cube $Q\in \dy$ with $\ell(Q)=2^{k+1}$ and $x\in \overline{Q}$.  Then, for $r\in (2^k,2^{k+1})$, $B(x,r)\subset B(x_Q, 2\sqrt{d}\ell(Q))$ and so, for an $s$-plane $L$,
$$\frac{1}{r^s}\int_{B(x,r)}\Bigl(\frac{\dist(y,L)}{r}\Bigl)^2d\mu(y)\leq 2^{s+2}\frac{1}{\ell(Q)^s}\int_{\R^d}\varphi_Q(y)\Bigl(\frac{\dist(y,L)}{\ell(Q)}\Bigl)^2d\mu(y),
$$
while also $\frac{\mu(B(x,r))}{r^s}\leq 2^sD_{\mu}(Q)$, and $\varphi_Q(x)=1$.  Thus the sum (\ref{splitcontjones}) is dominated by a constant multiple of
$$\sum_{k\in \mathbb{Z}}\sum_{Q\in \dy : \ell(Q)=2^{k+1}}\beta_{\mu}(Q)^2D_{\mu}(Q)^2\varphi_Q(x).
$$






\subsection{Lattice stabilization}\label{stabilize} We say that a sequence of dyadic lattices $\mathcal{D}^{(k)}$ stabilizes in a dyadic lattice $\mathcal{D}'$ if every $Q'\in \mathcal{D}'$ lies in $\mathcal{D}^{(k)}$ for sufficiently large $k$.

\begin{lem} Suppose $\mathcal{D}^{(k)}$ is a sequence of dyadic lattices with $Q_0\in \mathcal{D}^{(k)}$ for all $k$.  Then there exists a subsequence of the dyadic lattices that stabilizes to some dyadic lattice $\mathcal{D}'$.
\end{lem}

The lemma is proved via a diagonal argument:  For every $n\geq0$, there are $2^{nd}$ ways to choose a dyadic cube of sidelength $2^n$ so that $(-\tfrac{1}{2},\tfrac{1}{2})^d$ is one of its dyadic descendants.



\subsection{A basic density result}\label{basicdensity}

 For an integer $n$, a set $E$ is called \emph{$n$-rectifiable} if it is contained, up to an exceptional set of $\mathcal{H}^n$-measure zero, in the union of a countable number of images of Lipschitz mappings $f:\R^n\mapsto\R^d$.  We shall require the following elementary density property of measures supported on rectifiable sets,  whose proof may be found in Mattila \cite{Mat}.

\begin{lem}\label{rectdens}
Suppose that $\mu$ is a measure supported on an $n$-rectifiable set.  Then
$$\liminf\limits_{\substack{Q\ni x,\,\ell(Q)\rightarrow 0}}D_{\mu, n}(Q)>0 \text{ for }\mu\text{-almost every }x\in \R^d.
$$
\end{lem}

We shall actually only require this result when the support of $\mu$ is locally contained in a finite union of smooth $n$-surfaces. 



\subsection{The growth condition}\label{growth}    \begin{lem}Fix $s\in (0,d)$. If $\mu$ is a non-atomic measure for which the square function operator $\S_{\mu}$ is bounded in $L^2(\mu)$, then $\sup_{Q\in \mathcal{D}}D_{\mu}(Q)<\infty$ for any lattice $\mathcal{D}$.\end{lem}

 This lemma is well-known, and is essentially due to G. David.   Since we could only locate a proof in the case of non-degenerate Calder\'{o}n-Zygmund operators rather than the square function, we reproduce a sketch of David's argument (Proposition 1.4 in Chapter 3 of \cite{Dav}) in the context of the square function.  We shall verify that there is a constant $C>0$ such that for any cube $Q\subset \R^d$, $\mu(Q)\leq C\ell(Q)^s$, from which the lemma certainly follows (see (\ref{densequiv})). 

The first step is to use the pigeonhole principle to verify the following:

\textbf{Claim.}  For every integer $A>100$, there exists $C_0>0$, such that for any cube $Q\subset \R^d$, there exists a sub-cube $Q^*\subset Q$, with $\ell(Q^*) = \ell(Q)/A$, satisfying the property that
\begin{equation}\label{pigeoncube}\mu(Q^*) \geq \bigl(1-\frac{C_0}{\lambda^2}\Bigl)\mu(Q),
\end{equation}
where $\lambda = \frac{\mu(Q)}{\ell(Q)^s}$.

Set $\kap = \tfrac{1}{1000}$.   We first locate a cube $Q'\subset Q$ of side-length $\ell(Q')=2\kap A^{-1}\ell(Q)$ satisfying\footnote{The factor of $2$ in the sidelength here is due to the fact that our cubes are open.} $\mu(Q')\geq \kap^d A^{-d}\mu(Q)$. If the lemma fails to hold for a given $C_0>0$, then one can find $Q''\subset Q$, with $\ell(Q'') = \frac{\kap \ell(Q)}{A}$, $d(Q'', Q')\geq \tfrac{\ell(Q)}{5A}$ and satisfying $\mu(Q'')\geq \frac{C_0}{\lambda^2}\frac{\kap^d \mu(Q)}{A^d}$.

Notice that, if $f=\chi_{Q''}$, then have $\S_{\mu}(f)(x) \geq c(A,\kap)\frac{\mu(Q'')}{\ell(Q)^s}$ for $x\in Q'$.  Squaring this bound and integrating over $Q'$ yields that
$$\frac{\mu(Q')\mu(Q'')^2}{\ell(Q)^{2s}}\leq C(A,\kap)\mu(Q''), \text{ and hence } \frac{\mu(Q')\mu(Q'')}{\ell(Q)^{2s}}\leq C(A,\kap).
$$
Plugging in the lower bounds on the measures of $Q'$ and $Q''$ gives
$$\frac{C_0\mu(Q)^{2}}{\lambda ^2\ell(Q)^{2s}}\leq C(A,\kap), \text{ and hence }C_0\leq C(A,\kap).
$$
But this is absurd if $C_0$ was chosen large enough.  The claim is proved.

Starting with any cube $Q^{(0)}$, we iterate the claim to find a sequence of cubes $Q^{(j)}$, $j\geq 0$ with $Q^{(j)}\subset Q^{(j-1)}$, $\ell(Q^{(j)}) = \ell(Q^{(j-1)})/A$, and, with $\lambda^{(j)} = \frac{\mu(Q^{(j)})}{\ell(Q^{(j)})^s}$,
$$\lambda^{(j)}\geq A^s \Bigl(1-\frac{C_0}{(\lambda^{(j-1)})^2}\Bigl)\lambda^{(j-1)}.
$$
Assuming $\lambda^{(0)}\geq 1$ is large enough in terms of $C_0$, we infer by induction that $\lambda^{(j)}\geq A^{s/2}\lambda^{(j-1)}\geq \dots\geq A^{sj/2}\lambda^{(0)}$.  Plugging this back into (\ref{pigeoncube}) yields that for every $j$
$$\mu(Q^{(j)}) \geq \prod_{\ell =1}^{j-1}\Bigl(1-\frac{C_0}{A^{s\ell/2}(\lambda^{(0)})^2}\Bigl)\mu(Q^{(0)}).
$$
Assuming $\lambda^{(0)}$ is large enough, we have that $\mu(Q^{(j)})\geq \frac{1}{2}\mu(Q^{(0)})$ for every $j\geq 1$,
which implies that the non-atomic measure $\mu$ has an atom.  Consequently, there is an absolute bound $C>0$ for which $\lambda^{(0)}\leq C$.  Since $Q^{(0)}$ was an arbitrary cube, we have proved the desired growth condition on the measure.

\section{The basic scheme}

\subsection{Localization to square function constituents}


Let us now suppose that $\mu$ is a measure for which the square function operator $\mathcal{S}_{\mu}$ is bounded in $L^2(\mu)$.  For a dyadic lattice $\dy$, notice that for each $k\in \mathbb{Z}$, the balls $\{AB_{Q}: Q\in \mathcal{D}, \ell(Q)=2^k\}$ have overlap number at most $CA^d$.  Thus,
\begin{equation}\label{squaredecompf}\begin{split}\sum_{Q\in \dy}\int_{AB_Q} \int_{\tfrac{\ell(Q)}{A}}^{A\ell(Q)}\Bigl|\int_{\R^d}&\frac{x-y}{t^{s+1}}\varphi\Bigl(\frac{|x-y|}{t}\Bigl)f(y)d\mu(y)\Bigl|^2\frac{dt}{t}d\mu(x)\\&\leq C(A)\|\S_{\mu}\|^2_{L^2(\mu)\to L^2(\mu)}\|f\|_{L^2(\mu)}^2,
\end{split}\end{equation}
for every $f\in L^2(\mu)$.  Here $C(A) = CA^d\log(A)$,  as each $t\in (0,\infty)$ can lie in at most $C\log(A)$ of the intervals $[2^k/A, 2^kA]$, $k\in \mathbb{Z}$.  The precise form of $C(A)$ is not important.

We shall term the quantity \begin{equation}\label{sqfncoef}\mathcal{S}_{\mu}^A(Q) =\int_{AB_Q} \int_{\tfrac{\ell(Q)}{A}}^{A\ell(Q)}\Bigl|\int_{\R^d}\frac{x-y}{t^{s+1}}\varphi\Bigl(\frac{|x-y|}{t}\Bigl)d\mu(y)\Bigl|^2\frac{dt}{t}d\mu(x),\end{equation}
 a \emph{square function constituent}.  Our aim is to verify the following theorems.

\begin{thm}\label{wolffthm} If $s\not\in \mathbb{Z}$, then there are constants $C>0$ and $A>0$ such that for any measure $\mu$ satisfying $\sup_{Q\in \dy} D_{\mu}(Q)<\infty$, we have
\begin{equation}\label{wolffdom}
\sum_{Q\in \dy} D_{\mu}(Q)^2\I_{\mu}(Q)\leq C\sum_{Q\in \dy}\mathcal{S}^A_{\mu}(Q).
\end{equation}
\end{thm}

\begin{thm}\label{jonesthm} If $s\in \mathbb{Z}$, then there are constants $C>0$ and $A>0$ such that for any measure $\mu$ satisfying $\sup_{Q\in \dy} D_{\mu}(Q)<\infty$, we have
\begin{equation}\label{jonesdom}
\sum_{Q\in \dy} \beta_{\mu}(Q)^2D_{\mu}(Q)^2\I_{\mu}(Q)\leq C\sum_{Q\in \dy}\mathcal{S}^A_{\mu}(Q).
\end{equation}
\end{thm}


To see that Theorems \ref{thm2} and \ref{thm1} follow from Theorems \ref{jonesthm} and \ref{wolffthm} respectively, let us again assume that $\mu$ is a measure for which $\S_{\mu}$ is bounded on $L^2(\mu)$.  Then from Section \ref{growth} we see  that the condition $\sup_{Q\in \dy}D_{\mu}(Q)<\infty$ holds.  Fix a cube $P\in \dy$.  By testing the inequality (\ref{squaredecompf}) against the function $f=\chi_P$, we observe that the measure $\mu|_P$ satisfies
$$\sum_{Q\in \dy} \S_{\mu|_P}^A(Q)\leq C(A)\|S_{\mu}\|^2_{L^2(\mu)\to L^2(\mu)}\mu(P)
$$
for every $A>0$.  Now, from Theorems  \ref{jonesthm} and \ref{wolffthm} applied to $\mu|_P$, we find that if $s\in \mathbb{Z}$, then there is a constant $C>0$ such that
\begin{equation}\label{whatwewant}\sum_{Q\in \dy} \beta_{\mu|_P}(Q)^2D_{\mu|_P}(Q)^2\I_{\mu|_P}(Q)\leq C\|S_{\mu}\|^2_{L^2(\mu)\to L^2(\mu)}\mu(P),
\end{equation}
while, if $s\not\in \mathbb{Z}$, then then there is a constant $C>0$ such that
\begin{equation}\label{whatwewant2}\sum_{Q\in \dy} D_{\mu|_P}(Q)^2\I_{\mu|_P}(Q)\leq C\|S_{\mu}\|^2_{L^2(\mu)\to L^2(\mu)}\mu(P).
\end{equation}
Making reference to Section \ref{energies}, we conclude that the energy conditions (\ref{jonesenergy}) and (\ref{wolffenergy}) hold.

\subsection{The general principle that we will use over and over again}\label{principle}

Consider a rule $\Gamma$ that associates to each measure $\mu$ a function $$\Gamma_{\mu}:\mathcal{D}\rightarrow [0,\infty).$$

\begin{principle}  Fix $A>1$ and $\Delta>0$.  If we can verify the following statement:
\begin{gather}\text{for every measure }\mu\text{ and }Q\in \dy, \;\mathcal{S}_{\mu}^A(Q)\geq \Delta \Gamma_{\mu}(Q)\mathcal{I}_{\mu}(Q), \label{bigsquarecoef}
\end{gather}
then we get that
\begin{equation}\label{Gammasum}\sum_{Q\in \dy}\Gamma_{\mu}(Q)\mathcal{I}_{\mu}(Q)\leq \frac{1}{\Delta}\sum_{Q\in \dy}S_{\mu}^A(Q).
\end{equation}
\end{principle}

Comparing (\ref{Gammasum}) with (\ref{jonesdom}) and (\ref{wolffdom}), it is natural to attempt to verify (\ref{bigsquarecoef}) with the choice $$\Gamma_{\mu}(Q) = \begin{cases} \beta_{\mu}(Q)^2D_{\mu}(Q)^2  \text{ for }s\in \mathbb{Z},\\ \;\; D_{\mu}(Q)^2\text{ for }s\not\in \mathbb{Z}.\end{cases}$$  Unfortunately this is not possible.  As such, we shall use the general principle in a more convoluted way.

The key to proving Theorems \ref{jonesthm} and \ref{wolffthm} is to first understand the properties of measures for which no non-zero square function constituent can be found in any cube.  Following Mattila \cite{Mat, Mat2}, we call such measures $\varphi$-\emph{symmetric}.

\section{The structure of $\varphi$-symmetric measures}


A measure $\mu$ is called \emph{$\varphi$-symmetric} if
$$\int_{\R^d}(x-y)\varphi\Bigl(\frac{|x-y|}{t}\Bigl)d\mu(y)=0 \text{ for every }x\in \supp(\mu) \text{ and }t>0.
$$

We followed Mattila in the nomenclature: A measure is called symmetric if $\int_{B(x,r)}(x-y)d\mu(y)=0$ for every $x\in \supp(\mu)$ and $r>0$.  Of course this is a closely related object to the $\varphi$-symmetric measure, and we will lean heavily on the theory of symmetric measures developed by Mattila \cite{Mat2} and Mattila-Preiss \cite{MP}.

The reader may want to keep in mind the following example of a $\varphi$-symmetric measure: For a linear subspace $V$ of dimension $k\in \{0,\dots ,d\}$, a uniformly discrete set $E$ with $E\cap V=\{0\}$ that is symmetric about each of its points (that is, if $x\in E$, and $y\in E$, then $2y-x\in E$), and a non-negative symmetric function $f$ on $E$ (symmetry here means that if $x,y\in E$, then $f(x) = f(2y-x)$), form the measure
$$\mu = \sum_{x\in E}f(x)\mathcal{H}^k|_{V+x}.
$$
Then $\mu$ is $\varphi$-symmetric.  Provided that $\varphi$ is reasonably `non-degenerate', we expect that every $\varphi$-symmetric measure (with $0\in \supp(\mu)$) takes the above form, but we do not explore this too much here.

\subsection{Doubling scales}

Fix $\tau=1000\sqrt{d}$ and a constant $C_{\tau}>\tau^d$ to be chosen later. We shall call $R>0$ a doubling scale, or doubling radius, if
$$\mathcal{I}_{\mu}(B(0,\tau R))\leq C_{\tau}\mathcal{I}_{\mu}(B(0,R)).
$$

For $\lambda\in (0,\infty)$, we say that a measure has $\lambda$-power growth if
\begin{equation}\label{polygrowthdef}\limsup_{R\to\infty} \frac{\mu(B(0,R))}{R^{\lambda}}<\infty.
\end{equation}

\begin{lem}\label{doubexistence}  Suppose that $\mu$ is a measure with $\lambda$-power growth for some $\lambda\in (0,\infty)$. If $C_{\tau}>\tau^{\lambda}$,  then for every $R>0$, there is  a doubling scale $R'>R$.
\end{lem}

\begin{proof}  Since the statement is trivial if $\mu$ is the zero measure, we may assume that $\I_{\mu}(B(0,R))>0$. We consider radii of the form $\tau^kR$, $k\in \mathbb{N}$.  If none of these radii are doubling, then for every $k\in \mathbb{N}$ we have
\begin{equation}\begin{split}\nonumber\mathcal{I}_{\mu}(B(0,\tau^{k+1}R))&\geq C_{\tau}  \mathcal{I}_{\mu}(B(0,\tau^kR))\geq C_{\tau}^{k}\mathcal{I}_{\mu}(B(0,\tau R))\\&\geq C_{\tau}^{k}\mathcal{I}_{\mu}(B(0,R)).
\end{split}\end{equation}
But then as $C_{\tau}> \tau^{\lambda}$, we infer that
$$\lim_{k\to\infty}\frac{\mathcal{I}_{\mu}(B(0, \tau^k R))}{\tau^{k\lambda}}=\infty,
$$
which violates the growth condition (\ref{polygrowthdef}). Thus, under this condition on $C_{\tau}$, there exists some doubling scale $R'=\tau^kR$ with $k\geq 1$.\end{proof}

\subsection{Behaviour at infinity}








We next prove a variation of a powerful perturbation result used by Mattila-Preiss \cite{MP}.  

\begin{lem}[The Mattila-Preiss Formula]\label{MPformula}  Let $\mu$ be a $\varphi$-symmetric measure.  Suppose that $0\in\supp(\mu)$ and $x\in \supp(\mu)$.  Then,  whenever $R$ is a doubling radius with $R>|x|$,
$$\sup_{r\in [R,2R]}\Bigl|x+\frac{1}{\mathcal{I}_{\mu}(B(0,r))}\int_{\R^d}\frac{y}{r} \varphi'\Bigl(\frac{|y|}{r}\Bigl)\Bigl\langle\frac{y}{|y|},x\Bigl\rangle d\mu(y)\Bigl|\leq \frac{C_{\tau}C|x|^2}{R}
$$
\end{lem}

This formula does not appear precisely as stated in \cite{MP}.  The formulation is rather close to that of Lemma 8.2 in \cite{Tol4}, which in turn was strongly influenced by the techniques in \cite{MP}.


\begin{proof}  Since $\varphi\equiv 1$ on $[0,1]$, the function $\psi(x)=\varphi(|x|)$ lies in $C^{\infty}_0(B(0,2))$.  Taylor's theorem ensures that for each $y\in \R^d$,
\begin{equation}\label{MVT}
\varphi\Bigl(\frac{|x-y|}{r}\Bigl) = \varphi\Bigl(\frac{|y|}{r}\Bigl) - \frac{1}{r}\Bigl\langle x, \frac{y}{|y|}\Bigl\rangle \varphi'\Bigl(\frac{|y|}{r}\Bigl)+ \frac{E_{x,r}(y)}{2},
\end{equation}
where, for some $z$ on the line segment between $0$ and $x$,
$$E_{x,r}(y) = \frac{1}{r^2}\Bigl\langle x,D^2\psi\Bigl(\frac{z-y}{r}\Bigl)x\Bigl\rangle.
$$
Therefore, if $r>|x|$, then
$$|E_{x,r}(y)|\leq C\frac{|x|^2}{r^2}\chi_{B(0, 3r)}(y)\leq C\frac{|x|^2}{r^2}\varphi\Bigl(\frac{|y|}{3r}\Bigl).
$$




Now, since both $x$ and $0$ lie in $\supp(\mu)$, we have
$$\int_{\R^d}(x-y)\varphi\Bigl(\frac{|x-y|}{r}\Bigl)d\mu(y)=0,
$$
and also $\int_{\R^d}y\varphi\bigl(\frac{|y|}{r}\bigl)d\mu(y) = 0$.  Whence, for $r>|x|$,
\begin{equation}\begin{split}\nonumber&\Bigl|\int_{\R^d}x\varphi\Bigl(\frac{|y|}{r}\Bigl)d\mu(y)  -\int_{\R^d}(x-y)\frac{1}{r}\Bigl\langle \frac{y}{|y|},x\Bigl\rangle \varphi'\Bigl(\frac{|y|}{r}\Bigl) d\mu(y)\Bigl|\\ &\leq \int_{\R^d}|x-y||E_{x,r}(y)|d\mu(y)\leq C\frac{|x|^2}{r}\I_{\mu}(B(0, 3r)).
\end{split}\end{equation}
In conjunction with the straightforward estimate $$\int_{\R^d}\frac{|x|}{r}\Bigl|\Bigl\langle \frac{y}{|y|},x\Bigl\rangle\Bigl|\Bigl|\varphi'\Bigl(\frac{|y|}{r}\Bigl)\Bigl| d\mu(y)\leq \frac{C|x|^2}{r}\mu(B(0, 2r))\leq \frac{C|x|^2}{r}\I_{\mu}(B(0, 3r)),$$ we infer that
\begin{equation}\begin{split}\label{afterTaylor}&\Bigl|x\I_{\mu}(B(0,r))  +  \int_{\R^d}  \frac{y}{r} \varphi'\Bigl(\frac{|y|}{r}\Bigl)\Bigl\langle\frac{y}{|y|},x\Bigl\rangle d\mu(y)\Bigl| \leq \frac{C|x|^2}{r}\I_{\mu}(B(0, 3r)).
\end{split}\end{equation}
Finally, suppose $r\in [R,2R]$, with $R>|x|$ a doubling radius.  Then $\I_{\mu}(B(0,3r))\leq \I_{\mu}(B(0,\tau R))\leq C_{\tau}\I_{\mu}(B(0,r)).$ Thus, after dividing both sides of (\ref{afterTaylor}) by $\I_{\mu}(B(0,r))$, we arrive at the desired inequality.
\end{proof}

A variant of this formula was used in \cite{MP} to derive a growth rate at infinity of a symmetric measure.  We repeat their argument in the form of the following lemma, as we are working under different assumptions on the measure.

\begin{lem}[The growth lemma]\label{growthlem} Let $\mu$ be a $\varphi$-symmetric measure with $0\in \supp(\mu)$.  If $x_1, \dots, x_k$ is a maximal linearly independent set in $\supp(\mu)$, and $R$ is a doubling radius with $R>\max(|x_1|,\dots,|x_k|)$, then
$$ \sup_{r\in [R,2R]}\Bigl|k -r\frac{\tfrac{d}{dr}\mathcal{I}_{\mu}(B(0,r))}{\I_{\mu}(B(0,r))}\Bigl| \leq \frac{C(C_{\tau}, x_1,\dots,x_k)}{R}.
$$
\end{lem}

\begin{proof}  Consider the orthonormal basis $v_1,\dots, v_k$ of $V= \text{span}(\supp(\mu))$ obtained via the Gram-Schmidt algorithm from $x_1,\dots, x_k$.
 By applying Lemma \ref{MPformula} to each element $x_j$, and using the triangle inequality, we infer that, for every $j=1,\dots,k$,
\begin{equation}\label{basissmall}\sup_{r\in [R,2R]}\Bigl|v_j+\frac{1}{\I_{\mu}(B(0,r))}\int_{\R^d}\frac{y}{r} \varphi'\Bigl(\frac{|y|}{r}\Bigl)\bigl\langle\frac{y}{|y|},v_j\bigl\rangle d\mu(y)\Bigl|\leq \frac{C(C_{\tau}, x_1,\dots, x_k)}{R}.\end{equation}
But now observe that
\begin{equation}\begin{split}\label{longsplit}\sup_{r\in [R,2R]}&\Bigl|k +\sum_{j=1}^k \frac{1}{\I_{\mu}(B(0,r))}\int_{\R^d}\frac{\langle y,v_j\rangle}{r} \varphi'\Bigl(\frac{|y|}{r}\Bigl)\bigl\langle\frac{y}{|y|},v_j\bigl\rangle d\mu(y)\Bigl|\\
&\leq \sum_{j=1}^k\sup_{r\in [R,2R]}\Bigl|v_j+\frac{1}{\I_{\mu}(B(0,r))}\int_{\R^d}\frac{y}{r} \varphi'\Bigl(\frac{|y|}{r}\Bigl)\bigl\langle\frac{y}{|y|},v_j\bigl\rangle d\mu(y)\Bigl|\\&\leq \frac{C(C_{\tau}, x_1,\dots, x_k)}{R}.\end{split}\end{equation}
Finally, notice that since $v_1,\dots ,v_k$ form an orthonormal basis of $V$, we have
\begin{equation}\begin{split}\nonumber\sum_{j=1}^k\int_{\R^d}\frac{\langle y,v_j\rangle}{r} \varphi'\Bigl(\frac{|y|}{r}\Bigl)\bigl\langle\frac{y}{|y|},v_j\bigl\rangle d\mu(y) & = \int_{\R^d}\frac{|y|}{r}\varphi'\Bigl(\frac{|y|}{r}\Bigl)d\mu(y)\\& = -r\frac{d}{dr}\mathcal{I}_{\mu}(B(0,r)),
\end{split}\end{equation}
and the lemma follows by inserting this identity into the left hand side of (\ref{longsplit}).
\end{proof}

\begin{lem}[Maximal Growth at Infinity]\label{maxgrowth}  Let $\mu$ be a $\varphi$-symmetric measure with $0\in \supp(\mu)$.  Let $V$ denote the linear span of $\supp(\mu)$, and $k=\text{dim}(V)$.  Then for any $\eps>0$,
$$\liminf_{R\rightarrow \infty}\frac{\I_{\mu}(B(0,R))}{R^{k-\eps}}=+\infty.
$$
\end{lem}

\begin{proof}
From Lemma \ref{growthlem}, we may fix $R_0>0$ such that if $R\geq R_0$ is a doubling scale, then
\begin{equation}\label{smallRlarge}\sup_{r\in [R,2R]}\Bigl|k -r\frac{\tfrac{d}{dr}\mathcal{I}_{\mu}(B(0,r))}{\I_{\mu}(B(0,r))}\Bigl| \leq \frac{\eps}{2},\end{equation}
and so
$$ \frac{\tfrac{d}{dr}\mathcal{I}_{\mu}(B(0,r))}{\I_{\mu}(B(0,r))}\geq \frac{k-\frac{\eps}{2}}{r} \text{ for every }r\in [R, 2R].
$$
Integrating this inequality between $R$ and $2R$ yields that
$$\I_{\mu}(B(0,2R))\geq 2^{k-\frac{\eps}{2}}\I_{\mu}(B(0,R)).$$

We therefore infer the following alternative for \emph{any} $R\geq R_0$:  Either $R$ is a non-doubling radius, in which case, since $C_{\tau}\geq \tau^d$,
$$\I_{\mu}(B(0,\tau R))\geq C_{\tau}\I_{\mu}(B(0,R))\geq \tau^{k-\frac{\eps}{2}}\I_{\mu}(B(0,R)),
$$
or, $R$ is a doubling radius, in which case
$$\I_{\mu}(B(0,2R))\geq 2^{k-\frac{\eps}{2}}\I_{\mu}(B(0,R)).$$

Starting with $R_0$, we repeatedly apply the alternative to obtain a sequence of radii $R_j\to\infty$ with $R_{j}$ equal to either $2R_{j-1}$ or $\tau R_{j-1}$, such that $$\I_{\mu}(B(0,R_{j}))\geq \Bigl(\frac{R_{j}}{R_0}\Bigl)^{k-\frac{\eps}{2}}\I_{\mu}(B(0,R_0)).$$

Finally notice that for any $R\geq R_0$, there exists some $R_j$ with $\tfrac{R}{\tau}\leq R_j\leq R$, so
\begin{equation}\begin{split}\I_{\mu}(B(0,R))&\geq \I_{\mu}(B(0,R_j))\geq \Bigl(\frac{R_{j}}{R_0}\Bigl)^{k-\frac{\eps}{2}}\I_{\mu}(B(0,R_0))\\&\geq c\Bigl(\frac{R}{R_0}\Bigl)^{k-\frac{\eps}{2}}\I_{\mu}(B(0,R_0)).
\end{split}\end{equation}
The lemma is proved.\end{proof}

We shall need one additional corollary of the Mattila-Preiss formula.  It is a direct analogue for symmetric measures of an influential result of  Preiss (see Proposition 6.19 in \cite{DeLe}), which states that if a uniform measure is sufficiently flat at arbitrarily large scales (has small enough coefficient $\beta_{\mu,n}(Q)$ for all cubes $Q$ of sufficiently large side-length), then the measure is flat (supported in an $n$-plane).

 In the case of symmetric measures, this statement is much easier to achieve than for uniform measures due to the strength of the Mattila-Preiss formula\footnote{It is not true, though, that every symmetric measure is a uniform measure.}.  We give the statement in the contrapositive form as it will be convenient for our purposes.

\begin{lem}[Propagation of non-flatness to infinity]\label{propagation}  Let $\mu$ be a $\varphi$-symmetric measure with $0\in \supp(\mu)$.  Suppose that $\supp(\mu)$ is not contained in an $n$-plane.  There exists $R_{\mu}> 0$ such that if $R\geq R_{\mu}$ is a doubling scale, then
\begin{equation}\label{bigbetasym}\frac{1}{\I_{\mu}(B(0,R))}\inf_{L\in \mathcal{P}_n}\int_{\R^d}\Bigl(\frac{\dist(x,L)}{R}\Bigl)^2 \varphi\Bigl(\frac{|x|}{2R}\Bigl)d\mu(x)>\frac{1}{4C_{\tau}\|\varphi'\|_{\infty}^2}.
\end{equation}
\end{lem}

\begin{proof}[Proof of Lemma \ref{propagation}]  Since $0$ is the centre of mass of the measure $\varphi\bigl(\tfrac{\cdot}{2R}\bigl)d\mu$ ($\mu$ is symmetric and $0\in \supp(\mu)$), we infer from Lemma \ref{leastsquares} that it suffices to only consider linear subspaces $L$ in the infimum appearing on the left hand side of (\ref{bigbetasym}).


Set $V=\Span(\supp(\mu))$.  Then $V$ has dimension $k>n$ by the assumption of the lemma.  Notice that if $L$ is an $n$-dimensional linear subspace, then we have for every $y\in V$,
$$\dist(y,L)\geq \dist(y, L_V),
$$
where $L_V$ denotes the orthogonal projection of $L$ onto $V$.

Let $v_1,\dots, v_k$ be an orthonormal basis of $V$.  Then, using Lemma \ref{MPformula} in precisely the same manner as in the first paragraph of the proof of Lemma \ref{growthlem}, we find $R_{\mu}>0$ large enough so that for each $j=1,\dots, k$, and any doubling scale $R\geq R_{\mu}$,
\begin{equation}\label{MPlarge0} \Bigl|v_j+\frac{1}{R \I_{\mu}(B(0,R))}\int_{\R^d}\frac{y}{|y|}\varphi'\Bigl(\frac{|y|}{R}\Bigl)\bigl\langle y,v_j\bigl\rangle d\mu(y)\Bigl|<\frac{1}{2k}.
\end{equation}
 We can find a non-zero vector $x=\sum_{j=1}^k d_jv_j$ so that $x\perp L_V$.  Of course, $|x|^2=\sum_{j=1}^k |d_j|^2$ and so $|d_j|\leq |x|$ for every $j$.  Thus,
\begin{equation}\begin{split}\nonumber\Bigl|x+& \frac{1}{R\I_{\mu}(B(0,R))}\int_{\R^d}\frac{y}{|y|}\varphi'\Bigl(\frac{|y|}{R}\Bigl)\langle y,x\rangle d\mu(y)\Bigl|\\&\leq \sum_{j=1}^k |d_j| \Bigl|v_j+\frac{1}{R\I_{\mu}(B(0,R))}\int_{\R^d}\frac{y}{|y|}\varphi'\Bigl(\frac{|y|}{R}\Bigl)\langle y,v_j\rangle d\mu(y)\Bigl|\\
&\leq \sum_{j=1}^k |d_j|\frac{1}{2k}\leq \frac{|x|}{2}.
\end{split}\end{equation}

Consequently, we see that for any doubling scale $R> R_0$,
\begin{equation}\label{MPlarge}\frac{1}{2} < \Bigl|\frac{1}{\I_{\mu}(B(0,R))}\int_{B(0,2R)}\frac{y}{|y|}\varphi'\Bigl(\frac{|y|}{R}\Bigl)\frac{\bigl\langle y,\frac{x}{|x|}\bigl\rangle}{R} d\mu(y)\Bigl|,
\end{equation}
(here we have just used that $\varphi$ is supported in $B(0, 2R)$).  Now notice that, since $x\perp L_V$,  $|\bigl\langle y,\frac{x}{|x|}\bigl\rangle| \leq \dist(y,L_V)$.  Therefore, applying the Cauchy-Schwarz inequality to the right hand side of (\ref{MPlarge}), we get that
$$\frac{1}{2}\leq \|\varphi'\|_{\infty}\frac{\sqrt{\mu(B(0, 2R))}}{\I_{\mu}(B(0,R))}\Bigl(\int_{B(0, 2R)}\Bigl(\frac{\dist(y,L_V)}{R}\Bigl)^2d\mu(y)\Bigl)^{1/2}.
$$
The lemma now follows from the facts that $R$ is a doubling radius, and $\varphi\bigl(\frac{|y|}{2R}\bigl)=1$ for $y\in B(0, 2R)$.
\end{proof}

\subsection{Flat $\varphi$-symmetric measures}

We now look at the behaviour of a $\varphi$-symmetric measure that is supported in an $n$-plane.   Suppose that $\mu$ is a $\varphi$-symmetric measure with power growth (i.e. satisfies (\ref{polygrowthdef}) for some $\lambda>0$). Then, of course,
$$\int_{\R^d}(x-y)\Bigl[\varphi\Bigl(\frac{|x-y|}{t}\Bigl)-\varphi\Bigl(\frac{2|x-y|}{t}\Bigl)\Bigl]d\mu(y) = 0 \text{ for }x\in \supp(\mu),\, t>0.
$$
Notice that the function $t\mapsto \bigl[\varphi\bigl(\frac{1}{t}\bigl)-\varphi\bigl(\frac{2}{t}\bigl)\bigl]$ is supported in $[1/2, 2]$.  Consequently, if we take any bounded function $g:(0,\infty)\to\R$ that decays faster than any power at infinity, then for  $x\in \supp(\mu)$,
\begin{equation}\begin{split}\nonumber0&=\int_0^{\infty}g(t)\int_{\R^d}(x-y)\Bigl[\varphi\Bigl(\frac{|x-y|}{t}\Bigl)-\varphi\Bigl(\frac{2|x-y|}{t}\Bigl)\Bigl]d\mu(y)\frac{dt}{t}\\
& = \int_{\R^d}(x-y)\int_0^{\infty}g(|x-y|t)\Bigl[\varphi\Bigl(\frac{1}{t}\Bigl)-\varphi\Bigl(\frac{2}{t}\Bigl)\Bigl]\frac{dt}{t}d\mu(y).
\end{split}\end{equation}
  We shall use this idea to show that the support of a $\varphi$-symmetric measure is contained in the zero set of a real analytic function.  As usual, this idea goes back to Mattila \cite{Mat2}.

\begin{lem}\label{flatalt}  Suppose that $\mu$ is a $\varphi$-symmetric measure with power growth, and $\supp(\mu)\subset L$ for some $n$-plane $L$.  Then either $\mu = c_L\mathcal{H}^n|L$ for some $c_L>0$, or $\supp(\mu)$ is $(n-1)$-rectifiable.
\end{lem}

\begin{proof} After applying a suitable affine transformation, we may assume that $0\in \supp(\mu)$ and $L=\R^n\times \{0\}$, $0\in \R^{d-n}$.  

For $z\in \mathbb{C}^n$, set
$$w(z) = \int_0^{\infty}e^{-\pi z^2t^2}\Bigl[\varphi\Bigl(\frac{1}{t}\Bigl)-\varphi\Bigl(\frac{2}{t}\Bigl)\Bigl]\frac{dt}{t},
$$
where $z^2 = z_1^2+\dots+z_n^2$. Since the domain of integration may be restricted to $[1/2,2]$, we see that $w$ is an entire function on $\mathbb{C}^n$.  Consider the function $v:\mathbb{C}^n\to \mathbb{C}^n$ given by $v(z) = z w(z)$.  Then $v$ is an entire vector field.
 Notice that, with $\wh{v}(\xi) = \int_{\R^n}v(x)e^{-2\pi i \langle x,\xi\rangle}dm_n(\xi)$, $\xi\in \R^n$,  the Fourier transform of $v$ in $\R^n$, we have
\begin{equation}\begin{split}\wh{v}(\xi) &= c\nabla \wh{w}(\xi) = c\nabla \int_0^{\infty}\frac{1}{t^{n}}e^{-\pi|\xi|^2/t^2}\Bigl[\varphi\Bigl(\frac{1}{t}\Bigl)-\varphi\Bigl(\frac{2}{t}\Bigl)\Bigl]\frac{dt}{t}\\& = c\xi\int_0^{\infty}\frac{1}{t^{n+2}}e^{-\pi|\xi|^2/t^2}\Bigl[\varphi\Bigl(\frac{1}{t}\Bigl)-\varphi\Bigl(\frac{2}{t}\Bigl)\Bigl]\frac{dt}{t}.
\end{split}\end{equation}
The only thing we need from this formula is that $\wh v$ is only zero when $\xi=0$.

Since $\mu$ has power growth, and the entire function $v$ satisfies a straightforward decay estimate  $|v(x+iy)|\leq (1+|y|)e^{\pi|y|^2}(1+|x|)e^{-\pi|x|^2}$ for $x, y \in \R^n$, we infer that the function
$$u(x) = \int_{\R^n}v(x-y)d\mu(y), \, x\in \R^n,
$$
is a real analytic function on $\R^n$, and $\supp(\mu)\subset u^{-1}(\{0\})$.  First suppose that $u$ is identically zero in $\R^n$.   Then since $\mu$ is a tempered distribution\footnote{The power growth assumption is again used here.}, we have that
$$\wh{u} = c\wh v \cdot \wh\mu \equiv 0.
$$
Since $\wh v$ is only zero at the origin, we see that $\supp(\wh\mu) \subset \{0\}$.  This can only happen if $\mu$ has a polynomial density with respect to $m_n$, $\mu =Pm_n$. Since the function $\int_{\R^n}P(\,\cdot\,-y)v(y)dm_n(y)$ is identically zero on $\R^n$, we have that $\int_{\R^n}D^{\alpha}P(x-y)v(y)dm_n(y)= 0$ for any $x\in \R^n$ and multi-index $\alpha=(\alpha_1,\dots,\alpha_n)$.  But then if the polynomial $P$ is non-constant, we can, with a suitable differentiation, find a non-constant affine polynomial $\langle a, x\rangle+b$, $a\in \R^n$ and $b\in \R$, such that
$$\int_{\R^n} [\langle a,(x-y)\rangle+b]v(y)dm_n(y) =0 \text{ for every }x\in \R^n.
$$
Since $\int_{\R^n}v(y)dm_n(y) = \int_{\R^n}y w(y)dm_n(y)=0$, evaluating this expression at $x=0$, and taking the scalar product with $a$ yields
$$\int_{\R^n} \langle a, y\rangle ^2w(y)dm_n(y)=0,
$$
which is preposterous.  Consequently, $\mu$ is equal to a constant multiple of the Lebesgue measure $m_n$.


If $u$ is not identically zero, then since $u$ is analytic,
$$\R^n = \bigcup_{\alpha \text{ multi-index}}\{x\in \R^n: D^{\alpha}u(x)\neq 0\},
$$
and therefore
\begin{equation}\begin{split}\nonumber\supp(\mu) & \subset u^{-1}(\{0\})\cap\bigcup_{\alpha \text{ multi-index}}\{x\in \R^n: D^{\alpha}u(x)\neq 0\}\\
& =\!\!\!\!\! \bigcup_{\alpha \text{ multi-index}}\!\!\!\bigl\{x\in \R^n: D^{\alpha}u(x)\neq 0,\, D^{\beta}u(x)=0 \text{ for every }\beta<\alpha\bigl\}.
\end{split}\end{equation}
The implicit function theorem ensures that each set in the union on the right hand side is locally contained in a smooth $(n-1)$-surface.
\end{proof}

\subsection{A structure theorem}\label{symsum}

Here we summarize the results of this section in a form useful for what follows.

\begin{prop}\label{structureprop}Suppose that $\mu$ is a $\varphi$-symmetric measure satisfying $B_{Q_0}\cap \supp(\mu)\neq \varnothing$, and such that
$$\limsup_{R\to\infty}\frac{\mu(B(0,R))}{R^{\lambda}}<\infty \text{ for some }\lambda>0.
$$
Then
\begin{enumerate}
\item  If $\supp(\mu)$ is not contained in any $n$-plane, then

-- for any $\eps>0$ and every $T>1$, there exists $\ell>0$ such that if $Q$ is a cube satisfying $\ell(Q)\geq \ell$ and $\tfrac{1}{2}B_Q\supset B_{Q_0}$, then $D_{\mu,n+1-\eps}(Q)>T$.

-- there exists a constant $c^{\star}>0$, depending on $s$, $d$, $\lambda$, and $\|\varphi'\|_{\infty}$, such that whenever $\dy$ is a dyadic lattice and $\ell>0$, there exists $Q'\in \dy$ with $\ell(Q')\geq \ell$ satisfying $\frac{1}{2}B_{Q'}\supset B_{Q_0}$ and $$\beta_{\mu,n}(Q')\geq c^{\star}.$$
\item If $\supp(\mu)\subset L$ for some $n$-plane $L$, then either $\mu = c\mathcal{H}^n|L$ or $\supp(\mu)$ is $(n-1)$-rectifiable.
\end{enumerate}
\end{prop}

 \begin{proof} First assume that $\supp(\mu)$ is not contained in any $n$-plane.  Fix some point $x_0\in \supp(\mu)\cap Q_0$.  To prove the first property listed in item (1), observe that from Lemma \ref{maxgrowth} applied to the $\varphi$-symmetric measure $\mu_{x_0}= \mu(\, \cdot\, +x_0)$ it follows that $\lim_{R\to\infty}\frac{\I_{\mu}(B(x_0, R))}{R^{n+1-\eps}}=\infty$.  But if $\tfrac{1}{2}B_Q\supset B_{Q_0}$, then $B(x_Q, 2\sqrt{d}\ell(Q))$ contains a ball $B(x_0, R)$ with $R$ comparable to $\ell(Q)$. Then $D_{\mu,n+1-\eps}(Q)\geq c \frac{\I_{\mu}(B(x_0, R))}{R^{n+1-\eps}}$, and the first statement listed in item (1) follows.  

 To derive the second property listed in item (1), apply Lemma \ref{doubexistence} to the $\varphi$-symmetric measure $\mu_{x_0}= \mu(\, \cdot\, +x_0)$ to infer that, provided $C_{\tau}>\tau^{\lambda}$ (we fix $C_{\tau}$ to be of this order of magnitude), the measure $\mu_{x_0}$ has a sequence of doubling scales $R_j$ with $R_j\to\infty$.  Lemma \ref{propagation} yields that if $j$ is large enough, then $$\frac{1}{\I_{\mu}(B(x_0,R_j))}\inf_{L\in \mathcal{P}_n}\int_{B(x_0, 4R_j)}\Bigl(\frac{\dist(x,L)}{R_j}\Bigl)^2\varphi\Bigl(\frac{|x-x_0|}{2R_j}\Bigl)d\mu(x) \geq c,$$ for some constant $c>0$ depending on $s$, $d$, $\lambda$, and $\|\varphi'\|_{\infty}$.

 Now, for any given lattice $\dy$, choose a cube $Q$ intersecting $B(x_0, R_j)$ of side-length between $4R_j$ and $8R_j$,  For large enough $j$, we certainly have that $\frac{1}{2}B_Q\supset B_{Q_0}$.  Also notice that $B(x_0, 4R_j)\subset B(x_Q, 2\sqrt{d}\ell(Q))\subset \{\varphi_{Q}=1\}\subset \supp(\varphi_Q)\subset B(x_Q, 4\sqrt{d}\ell(Q))\subset B(x_0, \tau R_j)$.  Consequently, for any $n$-plane $L$,
\begin{equation}\begin{split}\nonumber\int_{B(x_0, 4R_j)}&\Bigl(\frac{\dist(x,L)}{R_j}\Bigl)^2\varphi\Bigl(\frac{|x-x_0|}{2R_j}\Bigl)d\mu(x)\\&\leq C \int_{B_Q}\Bigl(\frac{\dist(x,L)}{\ell(Q)}\Bigl)^2\varphi_Q(x)d\mu(x),
 \end{split}\end{equation}
 while $\mathcal{I}_{\mu}(Q)\leq C_{\tau}\I_{\mu}(B(x_0, R_j))$.  Bringing these observations together proves the second statement listed in item (1).



Item (2) is merely a restatement of Lemma \ref{flatalt}.\end{proof}

\section{The rudiments of weak convergence}

We say that a sequence of measures $\mu_k$ converges weakly to a measure $\mu$, written $\mu_k \wlim \mu$, if
$$\lim_{k\to\infty}\int_{\R^d}fd\mu_k = \int_{\R^d}fd\mu,
$$
for every $f\in C_0(\R^d)$ (the space of continuous functions on $\R^d$ with compact support). 

\subsection{A general convergence result}

Our first result is a simple convergence lemma that we shall use in blow-up arguments.

\begin{lem} \label{stupidlemma} Suppose that $\nu_k\wlim\nu$.  Fix $\psi\in \Lip(\R^d\times \R^d)$, and a sequence of functions $\psi_k\in \Lip(\R^d\times\R^d)$ such that
\begin{itemize}
\item $\psi_k$ converge uniformly to $\psi$,
\item there exists $R>0$ such that $\supp(\psi_k(x,\cdot))\subset B(x,R)$ for every $x\in \R^d$ and $k\in \mathbb{N}$, and
\item $\sup_k\|\psi_k\|_{\Lip}<\infty$.
\end{itemize}
Then, for any bounded open set $U\subset \R^d$, 
\begin{equation}\begin{split}\nonumber\liminf_{k\to\infty}&\int_{U}\Bigl|\int_{\R^d}\psi_k(x,y)d\nu_k(y)\Bigl|^2d\nu_k(x)\\&\geq \int_{U}\Bigl|\int_{\R^d}\psi(x,y)d\nu(y)\Bigl|^2d\nu(x)
\end{split}\end{equation}
\end{lem}

\begin{proof}
Choose $M$ such that $U\subset B(0, M)$.  Notice that the function
$$f_k(x) = \int_{\R^d}\psi_k(x,y)d\nu_k(y)
$$
has both its modulus of continuity and supremum norm on the set $B(0,M)$ bounded in terms of $M$, $R$, $\sup_k\|\psi_k\|_{\Lip}$ and $\sup_k\nu_k(B(0, R+M))$.   Consequently, the functions $f_k$ converge uniformly to the function $f(x)=\int_{\R^d}\psi(x,y)d\nu(y)$ on $B(0,M)$.  But now, for $g\in C_0(B(0,M))$, the sequence $g|f_k|^2$ converges to $g|f|^2$ uniformly, and so from the weak convergence of $\nu_k$ to $\nu$  we conclude that
$$\lim_{k\rightarrow\infty}\int_{\R^d}g|f_k|^2d\nu_k = \int_{\R^d}g|f|^2d\nu.
$$
The desired lower semi-continuity property readily follows by choosing for $g$ an increasing sequence of functions in $\Lip_0(B(0,M))$ that converges to $\chi_{U}$ pointwise.
\end{proof}





\begin{lem}\label{genconv}  Suppose that $\mu_k$ is a sequence of measures satisfying
\begin{enumerate}
\item $\I_{\mu_k}(Q_0)\geq 1$,
\item $\sup_k \mu_k(B(0,R))<\infty$ for every $R>0$,
\item $\mathcal{S}_{\mu_k}^k(Q_0)\leq \frac{1}{k}$.
\end{enumerate}
Then there is a subsequence of the measures that converges weakly to a $\varphi$-symmetric measure $\mu$ satisfying $\I_{\mu}(Q_0) \geq 1$.
\end{lem}

The reader should compare item (3) in the assumptions of the lemma with the display (\ref{bigsquarecoef}).  This lemma will be used to argue by contradiction that (\ref{bigsquarecoef}) holds for certain choices of function $\Gamma$.

\begin{proof}Using the condition (2) we pass to a subsequence of the measures that converges weakly to a measure $\mu$.  It is immediate from (1) that $\I_{\mu}(Q_0) \geq 1$. To complete the proof it remains to demonstrate that $\mu$ is $\varphi$-symmetric, that this,
\begin{equation}\label{genconvwhatwewant}\int_{\R^d}(x-y)\varphi\Bigl(\frac{|x-y|}{t}\Bigl)d\mu(y)=0\text{ for every }x\in \supp(\mu)\text{ and }t>0.
\end{equation}

To this end, fix $M>0$ and $t>0$.  We apply Lemma \ref{stupidlemma} with $\nu_k=\mu_k$, $\nu=\mu$, and $\psi_k(x,y)=(x-y)\varphi\bigl(\tfrac{|x-y|}{t}\bigl)$.  This yields that
\begin{equation}\begin{split}\nonumber\int_{B(0,M)}\Bigl|\int_{\R^d}&(x-y)\varphi\Bigl(\frac{|x-y|}{t}\Bigl) d\mu(y)\Bigl|^2d\mu(x)\\&\leq \liminf_{k\to\infty}\int_{B(0,M)}\Bigl|\int_{\R^d}(x-y)\varphi\Bigl(\frac{|x-y|}{t}\Bigl) d\mu_k(y)\Bigl|^2d\mu_k(x).
\end{split}\end{equation}
After dividing both sides by $\tfrac{1}{t^{2(s+1)}}$, integrating this inequality over $(\tfrac{1}{M},M)$ with respect to $\tfrac{dt}{t}$ and applying Fatou's lemma we get
$$\int_{B(0,M)}\int_{\tfrac{1}{M}}^M\Bigl|\int_{\R^d}\frac{x-y}{t^{s+1}}\varphi\Bigl(\frac{|x-y|}{t}\Bigl) d\mu(y)\Bigl|^2\frac{dt}{t}d\mu(x)\leq \liminf_k \mathcal{S}_{\mu_k, \varphi}^M(Q_0),$$
and the right hand side is equal to $0$ because of the condition (3) (just note that $\mathcal{S}_{\mu_k, \varphi}^M(Q_0)\leq \mathcal{S}_{\mu_k, \varphi}^k(Q_0)$ for $k>M$). Since $M$ was chosen arbitrarily, and certainly the function $x\mapsto \int_{\R^d}(x-y)\varphi\bigl(\frac{|x-y|}{t}\bigl) d\mu(y)$ is continuous, we conclude that (\ref{genconvwhatwewant}) holds.
\end{proof}

\subsection{Geometric properties of measures and weak convergence}

In blow-up arguments, we shall often consider a sequence of measures with a weak limit that is $\varphi$-symmetric.  The lemmas of this section will allow us to extract information about the eventual behaviour of the sequence of measures from our knowledge of the limit measure.

\begin{lem}\label{geomconv}  Suppose $\mu_k\wlim \mu$, and $Q$ is a cube with $\I_{\mu}(Q)>0$.   Then for any $n>0$,
\begin{itemize}
\item $\lim_{k\to\infty} D_{\mu_k, n}(Q)= D_{\mu, n}(Q)$,  while, for $n\in \mathbb{Z}$,
\item $\beta_{\mu,n}(Q)= \lim_{k\rightarrow \infty}\beta_{\mu_k,n}(Q)$,\; and,
\item $\alpha_{\mu,n}(Q)= \lim_{k\rightarrow \infty}\alpha_{\mu_k,n}(Q).$
\end{itemize}
\end{lem}

\begin{proof} The first item of course follows immediately from the definition of weak convergence.  For the convergence of the $\beta$-coefficients, observe that 
for any finite subset $\wt{\mathcal{P}}'$ of the family $\wt{\mathcal{P}}$ of $n$-planes that intersect $B_Q$, we have
$$\lim_{k\to\infty}\min_{L\in \wt{\mathcal{P}}'} \int_{B_Q}\varphi_Q\dist(x,L)^2d\mu_k(x) = \min_{L\in \wt{\mathcal{P}}'} \int_{B_Q}\varphi_Q\dist(x,L)^2d\mu(x).
$$
From this, the convergence of the $\beta$-coefficients follows from observing that the collection of functions $\varphi_Q\dist(\,\cdot\,,L)^2$, $L\in \wt{\mathcal{P}}$, is a relatively compact set in $C(\overline{B_Q})$; and every plane which contains the centre of mass of any of the measures $\varphi_Q\mu_k$ or $\varphi_Q\mu$ must also intersect $B_Q$ (since $B_Q$ is a convex set containing $\supp(\varphi_Q)$).

 We argue similarly in the case of the $\alpha$ numbers: In this case we observe that

\begin{itemize}
\item $\mathcal{P}_Q = \{L\cap \overline{B_Q}:\, L\in \mathcal{P}_n,\, L\cap \tfrac{1}{4}B_{Q}\neq \varnothing\}$ is relatively compact in the Hausdorff metric, while, for any constant $K>0$,
\item $\mathcal{F} = \{f\in \Lip_0(3B_Q): \|f\|_{\Lip}\leq \tfrac{1}{\ell(Q)}\}$ is a relatively compact subset in $C_0(\R^d)$ equipped with uniform norm.
\end{itemize}

 For any finite subsets $\mathcal{P}'_Q\subset\mathcal{P}_Q$, $\mathcal{F}'\subset \mathcal{F}$, we have
\begin{equation}\begin{split}\nonumber\lim_{k\to\infty}&\max_{f\in \mathcal{F}'}\min_{L\in \mathcal{P}'_Q} \int_{B_Q}\varphi_Qf(x)d(\mu_k-\vartheta_{\mu_k,L}\mathcal{H}^n|_L)(x) \\&= \max_{f\in \mathcal{F}'}\min_{L\in \mathcal{P}'_Q} \int_{B_Q}\varphi_Qf(x)d(\mu-\vartheta_{\mu,L}\mathcal{H}^n|_L)(x).
\end{split}\end{equation}
To complete the proof, just notice that for every $f\in \mathcal{F}$, the function
$$L\cap \overline{B_Q}\mapsto \int_{\R^d} f\varphi_Q d\mathcal{H}^n|_L
$$
is continuous in the Hausdorff metric with a modulus of continuity bounded in terms of $\|\varphi_Q\|_{\Lip}$, and $\ell(Q)$, while the functionals
$$f\mapsto \int_{\R^d}\varphi_{Q}fd\mu_k, \text{ and }f\mapsto \int_{\R^d}\varphi_{Q}fd\mu
$$
are continuous in the uniform norm with moduli of continuity bounded independently of $k$.  Since the numbers $\vartheta_{\mu_k, L}$ are uniformly bounded over $k$ and $L\cap \overline{B_Q}\in \mathcal{P}_Q$, the convergence of the $\alpha$-coefficients follows.
\end{proof}

The next result is a clear consequence of Lemma \ref{geomconv} (and also Section \ref{stabilize}), but it will be useful to state it explicitly.

\begin{cor}\label{upapprox}
Suppose that $\mu_k\wlim\mu$.  Fix a sequence of lattices $\dy^{(k)}$ that stabilize in a lattice $\dy'$, $n\in \mathbb{Z}\cap (0,d)$, and $m\in (0,d)$.  If, for a cube $Q'\in \dy'$, we have $\beta_{\mu,n}(Q')>\beta$ and $D_{\mu,m}(Q')>T$, then for all sufficiently large $k$, we have
$$Q'\in \dy^{(k)},\; \beta_{\mu_k,n}(Q')>\beta \text{ and }D_{\mu_k, m}(Q')>T.$$
\end{cor}

\begin{lem}\label{downapprox}  Fix $n\in \mathbb{N}$, $n<d$.  Suppose that $\mu_k\wlim\mu$, for some measure $\mu$ with $\I_{\mu}(Q_0) =  1$ for which $\supp(\mu)$ is $n$-rectifiable. Fix a sequence of lattices $\dy^{(k)}$, all containing $Q_0$, that stabilize in a lattice $\dy'$.  For any $\delta\in (0,1)$ and $\kap>0$, we can find a finite collection of cubes $Q_j$ such that
 \begin{enumerate}
 \item  $\ell(Q_j)\leq\kap$,
\item  $3B_{Q_j}$  are disjoint, and $3B_{Q_j}\subset 3B_{Q_0}$,
\end{enumerate} and, for all sufficiently large\footnote{This largeness threshold is purely qualitative.  It may depend on $\kap$, but also on the density properties of $\mu$, and the rate at which the lattices $\dy^{(k)}$ stabilize.} $k$,
\begin{enumerate}[resume]
\item  $Q_j\in \dy^{(k)}$,
\item  $\displaystyle D_{\mu_k, n+\delta}(Q')\leq \Bigl(\frac{\ell(Q_j)}{\ell(Q')}\Bigl)^{\delta/2}D_{\mu_k, n+\delta}(Q_j)$ for every $Q'\in \dy^{(k)}$ satisfying $B_{Q_j}\subset B_{Q'}\subset 300 B_{Q_0}$,
\item $\sum_{j}\I_{\mu_k}(Q_j) \geq \frac{1}{C}\I_{\mu_k}(Q_0).$
\end{enumerate}

\end{lem}

\begin{proof} From Lemma \ref{rectdens} we infer that, for any $\delta'\in (0,\delta/2)$
$$\lim_{\substack{ Q'\in \dy', x\in Q'\\\ell(Q')\rightarrow 0}}D_{\mu, n+\delta'}(Q')=\infty \text{ for }\mu\text{-a.e. }x\in \supp(\mu).
$$
Fix $T>0$.  Consider the maximal (by inclusion of the associated balls $B_{Q'}$) cubes $Q'\in \dy'$ with $B_{Q'}\subset 300B_{Q_0}$ that intersect $B_{Q_0}$ and satisfy $D_{\mu, n+\delta'}(Q')> T$.  If $T$ is sufficiently large then $\ell(Q')\leq \kap$, and certainly $3B_{Q'}\subset 3B_{Q_0}$, and so property (1), along with the second assertion in property (2), hold for the maximal cubes $Q'$.

For each maximal cube $Q'$ we have that
 \begin{equation}\label{upkcomplimit}D_{\mu, n+\delta}(Q'')\leq 2^{-(\delta-\delta')[Q'':Q']}D_{\mu, n+\delta}(Q')
\end{equation}
 for every $Q''\in \dy' $ satisfying \begin{equation}\label{upcubeconds}B_{Q'}\subset B_{Q''}\subset 300 B_{Q_0}\text{ and }B_{Q''}\cap B_{Q_0}\neq \varnothing.\end{equation}  As there are only finitely many $Q''$ satisfying (\ref{upcubeconds}), we have that for large enough $k$ (possibly depending on $Q'$)
 \begin{equation}\begin{split}\label{upkcomp}&D_{\mu_k,n+\delta}(Q'')\leq 2^{-\tfrac{\delta}{2}[Q'':Q']}D_{\mu_k, n+\delta}(Q')\\&\text{ for every }Q''\in \dy' \text{ satisfying }(\ref{upcubeconds}).\end{split}
\end{equation}

Now take a finite subcollection $\mathcal{G}$ of the maximal cubes with the property that
$\sum_{Q'\in \mathcal{G}}\I_{\mu}(Q')>\frac{1}{2}\I_{\mu}(Q_0)=\frac{1}{2}$ ($\mu$-almost every point in $B_{Q_0}$ is contained in a maximal cube).

If $k$ is sufficiently large, then every cube $Q'$ in the finite collection $\mathcal{G}$  satisfies (\ref{upkcomp}).  Moreover, since the lattices $\dy^{(k)}$ stabilize, we infer that if $k$ is large enough, then every $Q'\in \mathcal{G}$ and every $Q''\in \dy'$ satisfying (\ref{upcubeconds}) lies in $\dy^{(k)}$. It follows that properties (3) and (4) hold for every cube in $\mathcal{G}$ if $k$ is large enough.

Finally, using the Vitali covering lemma we choose a pairwise disjoint sub-collection $\{3B_{Q'_j}\}_j$ of the collection of balls $\{3B_{Q'}: Q'\in \mathcal{G}\}$ such that $\bigcup_{j} 15B_{Q'_j}\supset \bigcup_{Q'\in \mathcal{G}}3B_{Q'}.$   From (\ref{upkcomplimit}) we derive that $\mu(15B_{Q'_j})\leq C\I_{\mu}(Q'_j)$.  Whence
 $$\frac{1}{2} < \sum_j \mu(15B_{Q_j})\leq C\sum_j\I_{\mu}(Q'_j).
 $$
 Thus, as long as $k$ is large enough, we have that $\sum_j \I_{\mu_k}(Q'_j)\geq \frac{1}{C}$.  Consequently, the collection of cubes $(Q'_j)_j$ satisfies all of the desired properties. \end{proof}

\section{Domination from below}\label{downdom}

Fix $n=\uis-1$.

We introduce two parameters, $\eps\in (0,1)$ and $\delta\in (0,1)$, satisfying
$$n+2\delta+\eps<s, \text{ and }s+2\eps <\lis +1.
$$

\subsection{Domination from below} We introduce a filter on a dyadic lattice $\dy$ from \cite{JNRT} called domination from below.  Fix a measure $\mu$, and subsets $\mathcal{G},\mathcal{G}'\subset \dy$.

\begin{defn}\label{defndown}  We say that $Q\in \mathcal{G}$ \textit{is dominated from below by a (finite) bunch of cubes} $Q_j\in \mathcal{G}'$ if the following conditions hold:
\begin{enumerate}
\item $D_{\mu}(Q_j)> 2^{\eps  [Q:Q_j]}D_{\mu}(Q)$,
\item $3B_{Q_j}$ are disjoint,
\item $3B_{Q_j}\subset 3B_{Q}$,
\item $\displaystyle\sum_j D_{\mu}(Q_j)^2 2^{-2\eps [Q:Q_j]}\I_{\mu}(Q_j) > D_{\mu}(Q)^2 \I_{\mu}(Q).$
\end{enumerate}
We set $\gdown(\mathcal{G}'$) to be the set of all cubes $Q$ in $\mathcal{G}$ that cannot be dominated from below by a bunch of cubes in $\mathcal{G}'$ (except for the trivial bunch consisting of $Q$ itself in the case when $Q\in \mathcal{G}\cap \mathcal{G}'$. If $\mathcal{G}'=\mathcal{G}$, then we just write $\gdown$ instead of $\gdown(\mathcal{G})$.
\end{defn}

\begin{lem}\label{downwolff}  Suppose that $\sup_{Q\in \mathcal{G}}D_{\mu}(Q)<\infty$.  Then there exists $c(\eps)>0$  such that
$$\sum_{Q\in \gdown}D_{\mu}(Q)^2\I_{\mu}(Q) \geq c(\eps) \sum_{Q\in \mathcal{G}}D_{\mu}(Q)^2\I_{\mu}(Q).
$$
\end{lem}

\begin{proof}We start with a simple claim.

\textbf{Claim.}   Every $Q\in \mathcal{G}$ with $\I_{\mu}(Q)>0$ is dominated from below by a bunch of cubes $P_{Q,j}$ in $\gdown$.

To prove the claim we make two observations. The first is transitivity: if the bunch $Q_1, \dots, Q_N$ dominates $Q'\in \mathcal{G}$ from below, and if (say) $Q_1$ is itself dominated from below by a bunch $P_1, \dots, P_{N'}$, then the bunch $P_1, \dots, P_{N'}$, $ Q_2, \dots, Q_N$ dominates $Q'$.  The second observation is that there are only finitely many cubes $Q'$ that can participate in a dominating bunch for $Q$:  Indeed, each such cube $Q'$ satisfies $D_{\mu}(Q')\geq 2^{\eps  [Q:Q']}D_{\mu}(Q)$, and so
$$[Q:Q']\leq\frac{1}{\eps}\log_2\Bigl(\frac{\sup_{Q''\in \mathcal{G}}D_{\mu}(Q'')}{D_{\mu}(Q)}\Bigl).$$

With these two observations in hand, we define a partial ordering on the finite bunches of cubes $(Q_j)_j$ that dominate $Q$ from below:  For two different dominating bunches $(Q^{(1)}_j)_j$ and $(Q^{(2)}_j)_j$, we say that $(Q^{(1)}_j)_j \prec (Q^{(2)}_j)_j$ if for each ball $3B_{Q^{(1)}_j}$, we have $3B_{Q^{(1)}_j}\subset 3B_{Q^{(2)}_k}$ for some $k$.  Since there are only finitely many cubes that can participate in a dominating bunch, there may be only finitely many different dominating bunches of $Q$, and hence there is a minimal (according to the partial order $\prec$) dominating bunch $(P_{Q,j})_j$.  Each cube $P_{Q,j}$ must lie in $\gdown$.



Now write
\begin{equation}\begin{split}\nonumber\sum_{Q\in \mathcal{G}}& D_{\mu}(Q)^2\I_{\mu}(Q) \leq \sum_{Q\in \mathcal{G}}\sum_j D_{\mu}(P_{Q,j})^2\I_{\mu}(P_{Q,j})2^{-2\eps [Q:P_{Q,j}]}\\
& \leq \sum_{P\in \gdown}D_{\mu}(P)^2\I_{\mu}(P)\Bigl[\sum_{Q: 3B_Q\supset 3B_P}2^{-2\eps [Q:P]}\Bigl].
\end{split}\end{equation}
The inner sum does not exceed $\tfrac{C}{\eps }$, and the lemma follows.\end{proof}


\begin{lem}\label{stupidthing}  There exists $c(\eps)>0$  such that
$$\sum_{Q\in \mathcal{G}\backslash \gdown(\mathcal{G}')}D_{\mu}(Q)^2\I_{\mu}(Q) \leq C(\eps) \sum_{Q\in \mathcal{G}'}D_{\mu}(Q)^2\I_{\mu}(Q).
$$
\end{lem}

\begin{proof}For each $Q\in\mathcal{G}\backslash \gdown(\mathcal{G}')$, just pick a bunch of cubes $P_{Q,j}$ in $\mathcal{G}'$ that dominates $Q$ from below.  Then we merely repeat the final calculation of the previous proof:
\begin{equation}\begin{split}\nonumber\sum_{Q\in \mathcal{G}\backslash \gdown(\mathcal{G}')}& D_{\mu}(Q)^2\I_{\mu}(Q) \\&\leq \sum_{Q\in \mathcal{G}\backslash \gdown(\mathcal{G}')}\sum_j D_{\mu}(P_{Q,j})^2\I_{\mu}(P_{Q,j})2^{-2\eps [Q:P_{Q,j}]}\\
& \leq \sum_{P\in \mathcal{G}'}D_{\mu}(P)^2\I_{\mu}(P)\Bigl[\sum_{Q\in \dy: 3B_Q\supset 3B_P}2^{-2\eps [Q:P]}\Bigl],
\end{split}\end{equation}
and the lemma follows.
\end{proof}

The domination from below filter is used in what follows to preclude the possibility that the support of a measure in a cube $Q\in \mathcal{G}_{\text{down}}$ concentrates on a set of dimension smaller than $s$.  In particular, we shall use the following lemma:

\begin{lem}\label{downdomlem}  Suppose that $\mu_k\wlim \mu$,  where $\I_{\mu}(Q_0) =  1$ and $\supp(\mu)$ is $n$-rectifiable (recall that $n=\uis-1$).  Fix a sequence of lattices $\dy^{(k)}$, all containing $Q_0$, that stabilize in a lattice $\dy'$.  Provided that $\kap>0$ is chosen sufficiently small, for all sufficiently large $k$, the bunch of cubes $Q_j$ constructed in Lemma \ref{downapprox} dominates $Q_0$ from below in the sense of properties (1)--(4) of Definition \ref{defndown}.
\end{lem}

\begin{proof}
First notice that, by property (4) of the conclusion of Lemma \ref{downapprox}, we have that $D_{\mu_k, n+\delta}(Q_j)\geq D_{\mu_k, n+\delta}(Q_0)$, and so \begin{equation}\begin{split}\label{downapproxstep} D_{\mu_k}(Q_j)&=2^{(s-n-\delta) [Q_j:Q_0]}D_{\mu_k,n+\delta}(Q_j)\\&\geq2^{(s-n-\delta) [Q_j:Q_0]}D_{\mu_k,n+\delta}(Q_0)\\&= 2^{(s-n-\delta)[Q_j:Q_0]}D_{\mu_k}(Q_0),\end{split}\end{equation} as long as $k$ is large enough.
Therefore property (1) of Definition \ref{defndown} is satisfied, as $s-n-\delta>\eps$.  Since properties (2) and (3) of Definition \ref{defndown} clearly hold, it remains to verify that
$$\sum_j D_{\mu_k}(Q_j)^22^{-2\eps[Q_0:Q_j]}\I_{\mu_k}(Q_j) >D_{\mu_k}(Q_0)^2\mathcal{I}_{\mu_k}(Q_0).
$$
Since $D_{\mu}(Q_0)=1$, we have $D_{\mu_k}(Q_0)\geq \frac{1}{2}$ for large $k$ and hence from (\ref{downapproxstep}) we derive that \begin{equation}\begin{split}\nonumber \sum_j D_{\mu_k}(Q_j)^2&e^{-2\eps [Q_0:Q_j]}\I_{\mu_k}(Q_j)\\& \geq \frac{1}{4}\min_j 2^{2(s-n-\delta-\eps)[Q_0:Q_j]}\sum_j \I_{\mu_k}(Q_j).
\end{split}\end{equation}
Using properties (1) and (5) in the conclusion of Lemma \ref{downapprox}, the right hand side here is clearly at least $$\frac{1}{C}\min_j 2^{2(s-n-\delta-\eps)[Q_0:Q_j]}\I_{\mu_k}(Q_0)\geq\frac{1}{C}\kap^{-2(s-n-\delta-\eps)}\I_{\mu_k}(Q_0)$$  since $\delta<s-n-\eps$.  But the right hand side here is larger than $\I_{\mu_k}(Q_0)$ provided that $\kap$ is small enough.\end{proof}

\section{Cubes with lower-dimensional density control are sparse}

Recall that $n = \uis -1$.  Fix a measure $\mu$.  For $M\in \mathbb{N}$, consider the set $\dy_M(\mu)$ 
  of cubes $Q\in \dy$ such that $$D_{\mu, n+\delta}(Q')\leq D_{\mu, n+\delta}(Q)$$ whenever $Q'\in \dy$ satisfies $B_{Q'}\supset B_{Q}$ and $[Q':Q]\leq M$.\\

The aim of this section is to prove the following result:

\begin{prop}\label{lowdoubsparse}  There exists $M\in \mathbb{N}$, $A>0$ and $C>0$ such that if $\mu$ is a finite measure satisfying $\sup_{Q\in \dy}D_{\mu}(Q)<\infty$, then
$$\sum_{Q\in \dy_{M}(\mu)}D_{\mu}(Q)^2\I_{\mu}(Q)\leq C\sum_{Q\in \dy} \mathcal{S}_{\mu}^A(Q). $$
\end{prop}

To prove Proposition \ref{lowdoubsparse}, we shall use the domination from below filter with the subsets $\mathcal{G}=\mathcal{G}' = \dy_M(\mu)$.  Write $\dy_{M,\text{down}}(\mu)$ for the set of cubes in $\dy_M(\mu)$ that cannot be dominated from below.   Lemma \ref{downwolff} yields that,
$$\sum_{Q\in \dy_{M,\text{down}}(\mu)}D_{\mu}(Q)^2\I_{\mu}(Q) \geq c(\eps) \sum_{Q\in \dy_{M}(\mu)}D_{\mu}(Q)^2\I_{\mu}(Q).
$$ Consequently, referring to the general principle of Section \ref{principle}, we find that in order to prove Proposition \ref{lowdoubsparse}, we need to verify (\ref{bigsquarecoef}) with $$\Gamma_{\mu}(Q) = \begin{cases} D_{\mu}(Q)^2\text{ if } Q\in \dy_{M,\text{down}}(\mu),\\\;0\text{ otherwise.}\end{cases}$$  We formulate this precisely as the following lemma:

\begin{lem}\label{downcoef}There exists $A>0$, $\Delta>0$ and $M\in \mathbb{N}$ such that for every measure $\mu$ and every cube $Q\in  \dy_{M,\operatorname{down}}(\mu)$, we have
$$\mathcal{S}_{\mu}^{A}(Q)\geq \Delta D_{\mu}(Q)^2\I_{\mu}(Q).
$$
\end{lem}

\begin{proof}  If the result fails to hold, then for every $k\in \mathbb{N}$, we can find a measure $\wt\mu_k$ and a cube $Q_k\in\dy_{k,\text{down}}(\wt\mu_k)$ such that
$$\mathcal{S}_{\wt\mu_k}^k(Q_k)\leq \frac{1}{k}D_{\wt\mu_k}(Q_k)^2\mathcal{I}_{\wt\mu_k}(Q_k).
$$

Consider the measure $\mu_k=\frac{\wt\mu_k(\mathcal{L}_{Q_k}(\,\cdot\,))}{\I_{\wt\mu_k}(Q_k)}$.  Then $D_{\mu_k}(Q_0)=\I_{\mu_k}(Q_0)=1$.  The preimage of the  lattice $\dy$  under the affine map $\mathcal{L}_{Q_k}$ is some lattice $\dy^{(k)}$ with $Q_0\in \dy^{(k)}$.  Of course, we have that $Q_0\in {\dy^{(k)}}_{k,\text{down}}(\mu_k)$. In addition
\begin{equation}\label{prelimdoub}D_{\mu_k,n+\delta}(Q')\leq 1 \text{ if } Q'\in \dy^{(k)}, \, B_{Q'}\supset B_{Q_0},\, [Q':Q_0]\leq k.
\end{equation}
It readily follows from this that for every $R>0$, $\sup_k \mu_k(B(0,R))<\infty$.   In addition, we have that
$\mathcal{S}_{\mu_k}^k(Q_0)\leq \frac{1}{k}.$
As such, we may apply Lemma \ref{genconv} and conclude that, passing to a subsequence if necessary, the sequence $\mu_k$ converges weakly to a $\varphi$-symmetric measure $\nu$ with $\I_{\nu}(Q_0)=1.$
With the passage to a further subsequence, we may assume that the lattices $\dy^{(k)}$ (which all contain $Q_0$) stabilize in a lattice $\dy'$ (see Section \ref{stabilize}).   Then from (\ref{prelimdoub}) we see that
$$D_{\nu, n+\delta}(Q')\leq D_{\nu, n+\delta}(Q_0) \text{ for every }Q'\in \dy' \text{ such that }  B_{Q'}\supset B_{Q_0}.
$$
From this property, we infer from Proposition \ref{structureprop} that $\supp(\nu)$ has to be contained in an $n$-plane and is therefore $n$-rectifiable ($\nu$ has insufficient growth at infinity for the other possibilities in Proposition \ref{structureprop} to hold).

 Consequently all the hypotheses of Lemma \ref{downapprox} are satisfied, and so we may consider the finite collection of cubes $Q_j$ constructed there, which may have side-length smaller than any prescribed threshold $\kap>0$.  Lemma \ref{downdomlem} ensures that the bunch of cubes $Q_j$ dominate $Q_0$ from below as long as $\kap$ is small enough.  Therefore, given that $Q_0\in {\dy^{(k)}}_{k,\text{down}}(\mu_k)$, we will have reached our desired contradiction once we verify the following:

 \textbf{Claim.} \emph{Provided that $\kap$ is small enough, and $k$ is sufficiently large, each cube $Q_j$ lies in ${\dy^{(k)}}_{k}(\mu_k)$.}

 To see this, notice that, for a cube $Q''$ with $B_{Q''}\supset B_{Q_j}$ and $\ell(Q'')\leq 2^k\ell(Q_j)$, it can only happen that property (4) of Lemma \ref{downapprox} does not immediately show that
$D_{\mu_k,n+\delta}(Q'')\leq D_{\mu_k, n+\delta}(Q_j)$ in the case when $B_{Q''}$ is not contained in $300B_{Q_0}$. But then since $B_{Q_j}\cap B_{Q_0}\neq \varnothing$, the cube $Q''$ has big side-length (certainly at least $30\ell(Q_0)$).   It follows that the grandparent of $Q''$, say $\wt{Q}''$, must satisfy $B_{\wt Q''}\supset B_{Q_0}$, while certainly $[\wt Q'':Q_0]\leq k$ if $\kap$ is small enough.  But now  we derive that
\begin{equation}\begin{split}\nonumber D_{\mu_k,n+\delta}(Q'')& \leq CD_{\mu_k, n+\delta}(\wt Q'')\stackrel{(\ref{prelimdoub})}{\leq} CD_{\mu_k, n+\delta}(Q_0)\\&\stackrel{\text{Lemma }\ref{downapprox},\, (4)}{\leq} C\ell(Q_j)^{\delta/2}D_{\mu_k, n+\delta}(Q_j)\leq D_{\mu_k, n+\delta}(Q_j)
\end{split}\end{equation}
 as long as $C\kap^{\delta/2}<1$.  The claim follows.
\end{proof}

\section{Domination from above and the proof of Theorem \ref{thm2}}

Consider a lattice $\dy$, and a non-negative function $\Ups:\dy\to [0,\infty)$.

\subsection{Domination from above}  We say that $Q'\in \dy$ dominates $Q\in\dy$ from above if $\tfrac{1}{2}B_{Q'}\supset B_{Q}$ and
$$\Ups(Q')>2^{\eps[Q':Q]}\Ups(Q)
$$
We let $\dyup$ denote the set of cubes $Q\in \dy$ that are not dominated from above by a cube in $\dy$.

\begin{lem}\label{upcubescontribute}  If $\sup_{Q\in \dy} \Ups(Q)<\infty$, then
$$\sum_{Q\in \dyup}\Ups(Q)^2\I_{\mu}(Q)\geq c(\eps)\sum_{Q\in \dy}\Ups(Q)^2\I_{\mu}(Q).
$$
\end{lem}

\begin{proof} We first claim that every $Q\in \dy\backslash \dyup$ with $\Upsilon(Q)>0$ can be dominated from above by a cube $\wt{Q}\in \dyup$.

Indeed, note that if $Q'$ dominates $Q$ from above, then certainly
$$[Q':Q]\leq\frac{1}{\eps}\log_2\Bigl(\frac{\sup_{Q''\in \mathcal{D}}\Upsilon(Q'')}{\Upsilon(Q)}\Bigl),$$ or else we would have that $\Upsilon(Q')>\sup_{Q''\in \mathcal{D}}\Upsilon(Q'')$ (which is absurd).  Consequently, there are only finitely many candidates for a cube that dominates $Q$ from above.  To complete the proof of the claim, choose $\wt{Q}\in \dy$ to be a cube of largest side-length that dominates $Q$ from above.  Then $\wt{Q}\in \dyup$ (domination from above is transitive).

For each fixed $P\in \dyup$, consider those $Q\in \dy\backslash \dyup$ with $\I_{\mu}(Q) >0$ and $\wt{Q}=P$.  Then
\begin{equation}\begin{split}\nonumber\sum_{Q\in \dy\backslash \dyup: \, \widetilde{Q}=P}& \Upsilon(Q)^2\I_{\mu}(Q) = \sum_{m\geq 1} \sum_{\substack{Q\in \dy\backslash \dyup:\\\ell(Q)=2^{-m}\ell(P),\, \widetilde{Q}=P}} \Upsilon(Q)^2\I_{\mu}(Q)\\
&\leq \sum_{m\geq 1}2^{-2\eps  m}\Upsilon(P)^2\Bigl[\sum_{\substack{Q\in \dy:\\\ell(Q)=2^{-m}\ell(P), B_Q\subset \tfrac{1}{2}B_P}} \I_{\mu}(Q)\Bigl].
\end{split}\end{equation}
The sum in square brackets is bounded by $C\I_{\mu}(P)$ (as $\varphi_P\equiv 1$ on $\tfrac{1}{2}B_P$), and so by summing over $P\in \dyup$, we see that
\begin{equation}\nonumber \sum_{Q\in \dy\backslash \dyup}\Upsilon(Q)^2\mathcal{I}_{\mu}(Q)\leq C(\eps)\sum_{P\in\dyup}\Upsilon(P)^2\I_{\mu}(P),\end{equation}
and the lemma is proved.
\end{proof}

We shall make the choice
$$\Ups(Q) = \begin{cases}\beta_{\mu}(Q)D_{\mu}(Q) \text{ if }s\in \mathbb{Z},\\
\;\; D_{\mu}(Q) \text{ if } s\not\in \mathbb{Z}.\end{cases}
$$
With this function, denote by $\dyup(\mu)$ those cubes that cannot be dominated from above.

Notice that if $s\in \mathbb{Z}$, and $Q\in \dyup(\mu)$, then for every $Q'\in \dy$ with $\tfrac{1}{2}B_{Q'}\supset B_{Q}$,
\begin{equation}\label{uprestate}\beta_{\mu}(Q')D_{\mu}(Q')\leq2^{\eps[Q':Q]}\beta_{\mu}(Q)D_{\mu}(Q).\end{equation}
Provided that $\beta_{\mu}(Q)>0$, we readily  derive from this inequality  that whenever $\tfrac{1}{2}B_{Q'}\supset B_Q$
\begin{equation}\begin{split}\label{densbetadoub}
 &\Bigl(\frac{\ell(Q)}{\ell(Q')}\Bigl)^sD_{\mu}(Q) \leq  D_{\mu}(Q')\leq \Bigl(\frac{\ell(Q')}{\ell(Q)}\Bigl)^{s+2\eps}D_{\mu}(Q),\\
 &\;\;\;\;\;\;\;\;\;\;\text{and }\beta_{\mu}(Q')\leq \Bigl(\frac{\ell(Q')}{\ell(Q)}\Bigl)^{s+\eps}\beta_{\mu}(Q).
 \end{split}\end{equation}
 The right hand inequality in the first displayed formula perhaps deserves comment.  To see it, we plug the obvious inequality $$\sqrt{\I_{\mu}(Q')}\beta_{\mu}(Q') \geq \sqrt{\I_{\mu}(Q)}\beta_{\mu}(Q)$$ into (\ref{uprestate}) to find that
 $$\frac{\sqrt{\I_{\mu}(Q')}}{\ell(Q')^s} \leq 2^{\eps[Q':Q]}\frac{\sqrt{\I_{\mu}(Q)}}{\ell(Q)^s}.
 $$
 Rearranging this yields the desired inequality.\\

If instead $s\not\in \mathbb{Z}$ and $Q\in \dyup(\mu)$ then we have much better density control:
$$D_{\mu}(Q')\leq \Bigl(\frac{\ell(Q')}{\ell(Q)}\Bigl)^{\eps}D_{\mu}(Q) \text{ whenever }Q'\in \dy, \tfrac{1}{2}B_{Q'}\supset B_Q.
$$

For the remainder of the paper, let us fix $M$ so that Proposition \ref{lowdoubsparse} holds.  Our goal will be to prove the following alternative.

\begin{alt}\label{nonflatalt}  For each $\Lambda>4$ and $\alpha>0$,  there exist $A>0$ and $\Delta>0$ such that for every measure $\mu$ and cube $Q\in \dyup(\mu)$, with the additional properties that $\beta_{\mu}(Q)>0$ and $\alpha_{\mu}(\Lambda Q)\geq \alpha$ if $s\in \mathbb{Z}$, we have that either

(a)  $\displaystyle\mathcal{S}_{\mu}^A(Q)\geq \Delta D_{\mu}(Q)^2\mathcal{I}_{\mu}(Q)$

or

(b)  $Q$ is dominated from below by a bunch of cubes in $\dy_{M}(\mu)$.
\end{alt}



Before we prove the alternative, let us see how we shall use it.  Fix $\Lambda>0$ and $\alpha>0$.  For $s\in \mathbb{Z}$, set
$$\dyup^{\star}(\mu) = \{Q\in \dyup(\mu)\,:\, \alpha_{\mu}(\Lambda Q)\geq \alpha\text{ and }\beta_{\mu}(Q)>0\},$$
while for $s\not\in \mathbb{Z}$, set $\dyup^{\star}(\mu) = \dyup(\mu)$.

\begin{cor}\label{nonflataltcor} If $s\in \mathbb{Z}$, then there exists $\Delta>0$ and $A>0$, depending on $M,$ $\Lambda,$ $\alpha$ such that for every finite measure $\mu$ satisfying $\sup_{Q\in \dy}D_{\mu}(Q)<\infty$,
$$\sum_{Q\in\dyup^{\star}(\mu)} D_{\mu}(Q)^2\I_{\mu}(Q)\leq \frac{1}{\Delta}\sum_{Q\in \dy}\mathcal{S}_{\mu}^A(Q).
$$
If $s\not\in \mathbb{Z}$, then there exists $\Delta>0$ and $A>0$, depending on $M$, such that for every finite measure $\mu$ satisfying $\sup_{Q\in \dy}D_{\mu}(Q)<\infty$,
$$\sum_{Q\in\dyup(\mu)} D_{\mu}(Q)^2\mathcal{I}_{\mu}(Q)\leq \frac{1}{\Delta}\sum_{Q\in \dy}\mathcal{S}_{\mu}^A(Q).
$$
\end{cor}

\begin{proof}[Proof of Corollary \ref{nonflataltcor}]

One uses the general principle (\ref{bigsquarecoef}) to control the contribution of the sum over the cubes where alternative (a) occurs.  Indeed, making the choice $$\Gamma_{\mu}(Q) = \begin{cases}D_{\mu}(Q)^2 \text{ if }Q\in \dyup^{\star}(\mu)\text{ and alternative }(a) \text{ holds for }Q\\
\;\;0\text{ otherwise},\end{cases}$$
we get from (\ref{Gammasum}) that
$$\sum_{\substack{Q\in \dyup^{\star}(\mu):\\
\text{ alternative }(a) \text{ holds}}}D_{\mu}(Q)^2\I_{\mu}(Q)\leq \frac{1}{\Delta}\sum_{Q\in \dy}\mathcal{S}_{\mu}^A(Q).
$$
 For the cubes where alternative (b) holds, we apply  Lemma \ref{stupidthing} with $\mathcal{G}' = \dy_M(\mu)$ and $\mathcal{G} = \dyup^{\star}(\mu)$.  Since $$\{Q\in \dyup^{\star}(\mu):
\text{ alternative }(b) \text{ holds}\}\subset \mathcal{G}\backslash \mathcal{G}_{\text{down}}(\mathcal{G}'),$$ we infer that
$$\sum_{\substack{Q\in \dyup^{\star}(\mu):\\
\text{ alternative }(b) \text{ holds}}}D_{\mu}(Q)^2\mathcal{I}_{\mu}(Q) \leq C(\eps)\sum_{Q\in \dy_M(\mu)}D_{\mu}(Q)^2\mathcal{I}_{\mu}(Q).
$$
Proposition \ref{lowdoubsparse} ensures that the right hand side here is bounded by the sum of square function constituents.
\end{proof}

Notice that, in conjunction with Lemma \ref{upcubescontribute}, \emph{Corollary \ref{nonflataltcor} completes the proof of Theorem \ref{wolffthm}, and with it Theorem \ref{thm2}.}\\

We now move onto proving the alternative.

\begin{proof}[Proof of Alternative \ref{nonflatalt}] We (rather predictably) proceed by contradiction. If the alternative fails to hold, then for some $\Lambda>0$ and $\alpha>0$, and every $k\in \mathbb{N}$, we can find a measure $\wt\mu_k$ and a cube $Q_k\in \dyup^{\star}(\wt\mu_k)$ such that
$$\mathcal{S}_{\wt\mu_k}^k(Q_k)\leq \frac{1}{k}D_{\wt\mu_k}(Q_k)^2\I_{\wt\mu_k}(Q_k),
$$
but also $Q_k$ cannot be dominated from below by a bunch of cubes in $\dy_{M}(\wt\mu_k)$.

We consider the measure $\mu_k=\frac{\wt\mu_k(\mathcal{L}_{Q_k}(\,\cdot\,))}{\I_{\wt\mu_k}(Q_k)}$, which satisfies $D_{\mu_k}(Q_0)=1$.  The preimage of the lattice $\dy$ under $\mathcal{L}_{Q_k}$ is some lattice $\dy^{(k)}$ with $Q_0\in \dy^{(k)}$.  Moreover, $Q_0\in \dyup^{(k)\,\star}(\mu_k)$, and so from (\ref{densbetadoub}) we have that
$$D_{\mu_k}(Q')\leq C\ell(Q')^{s+1} \text{ whenever }\frac{1}{2}B_{Q'}\supset B_{Q_0}.
$$

This polynomial growth bound allows us to apply Lemma \ref{genconv} and pass to a subsequence of the measures that converges weakly to a $\varphi$-symmetric measure $\mu$ with $D_{\mu}(Q_0)= \I_{\mu}(Q_0)=1$.  With the passage to a further subsequence, we assume that the lattices $\dy^{(k)}$ stabilize in some lattice $\dy'$.

We first suppose that $\supp(\mu)$ is not contained in an $\lis$-plane.  Then, since $\mu(B(0, R))\leq CR^{2s+2\eps}$ for large $R>0$, we may apply Proposition \ref{structureprop}, and find that there exists a cube $Q'\in \dy'$ with $\frac{1}{2}B_{Q'}\supset B_{Q_0}$ of arbitrarily large side-length we have that $\beta_{\mu}(Q')> c^{\star}$ and $D_{\mu}(Q')> \ell(Q')^{\lis+1-s-\eps}$, where $c^{\star}>0$ depends only on $d$, $s$, and $\|\varphi'\|_{\infty}$.  But Lemma \ref{upapprox} then ensures that for sufficiently large $k$, we have that $Q'\in \dy^{(k)}$, $\beta_{\mu_k}(Q')> c^{\star}$, and $D_{\mu_k}(Q')> \ell(Q')^{\lis+1-s-\eps}D_{\mu_k}(Q_0)$.  Provided that $c^{\star}\ell(Q')^{\lis +1-s-\eps}>4\sqrt{d}\cdot \ell(Q')^{\eps}\; (\geq \beta_{\mu_k}(Q_0)2^{\eps[Q':Q_0]})$, this contradicts the fact that $Q_0\in \dyupk(\mu_k)$, and such a contradictory choice of $\ell(Q')$ is possible since $2\eps<1+\lis-s$.  Thus $\supp(\mu)\subset L$ for some $\lis$-plane $L$ (which must intersect $B_{Q_0}$).

Our next claim is that $\supp(\mu)$ is $n$-rectifiable, with $n=\uis -1$.  This is already proved in the case when $s\not\in \mathbb{Z}$, as $\mu$ is supported in an $n$-plane in this case.  If $s\in \mathbb{Z}$, then we notice that Proposition \ref{structureprop} guarantees that either $\mu=c\mathcal{H}^{s}_L$, or that $\supp(\mu)$ is $n$-rectifiable.  But the first case is ruled out since $\alpha_{\mu}(\Lambda Q_0)>0$ (note that $\Lambda>4$), so we must indeed have that $\supp(\mu)$ is $n$-rectifiable.

  Consequently, we may apply Lemma \ref{downapprox} with $\kap\leq 2^{-M}$ to find a finite collection of cubes $Q_j$, each of sidelength less that $2^{-M}$, such that the balls $3B_{Q_j}$ are disjoint, $3B_{Q_j}\subset 3B_{Q_0}$, and for all sufficiently large $k$ we have, for every $j$, $$D_{\mu_k,n+\delta}(Q')\leq D_{\mu_k,n+\delta}(Q_j)$$ whenever $Q'\in \dy^{(k)}$ with $B_{Q_j}\subset B_{Q'}\subset 300B_{Q_0}$ (this property is weaker than property (4) of the conclusion of Lemma \ref{downapprox}).  In particular this ensures that each $Q_j$ lies in ${\dy^{(k)}}_{M}(\mu_k)$.

On the other hand, by choosing $\kap$ smaller if necessary, we conclude from Lemma \ref{downdomlem} that the bunch $Q_j$ dominates $Q_0$ from below.  This contradicts the fact that $Q_0$ cannot be dominated from below by a finite bunch of cubes from ${\dy^{(k)}}_M(\mu_k)$, and with this final contradiction we complete the proof of the alternative.\end{proof}

\section{The reduction to one last square function estimate}

For the remainder of the paper, we restrict our attention to proving Theorem \ref{jonesthm}, so we shall henceforth assume that  $s\in \mathbb{Z}$. It remains to show that there exist constants $\alpha>0$, $\Lambda>4$, $A>1$ and $C>0$, such that if $\mu$ is a finite measure satisfying $\sup_{Q\in \dy}D_{\mu}(Q)<\infty$, then
$$\sum_{Q\in\dyup(\mu),\, \alpha_{\mu}(\Lambda Q)\leq \alpha} \beta_{\mu}(Q)^2 D_{\mu}(Q)^2\mathcal{I}_{\mu}(Q)\leq C\sum_{Q\in \dy} \mathcal{S}_{\mu}^A(Q).
$$
When combined with Corollary \ref{nonflataltcor}, this would show that (with a possibly larger constant $A$),
$$\sum_{Q\in\dyup(\mu)} \beta_{\mu}(Q)^2 D_{\mu}(Q)^2\mathcal{I}_{\mu}(Q)\leq C\sum_{Q\in \dy} \mathcal{S}_{\mu}^A(Q).
$$
Then Theorem \ref{jonesthm} follows from Lemma \ref{upcubescontribute}.  Following the general principle (\ref{bigsquarecoef}) with the choice
\begin{equation}\nonumber\Gamma_{\mu}(Q) = \begin{cases}\beta_{\mu}(Q)^2D_{\mu}(Q)^2 \text{ if }Q\in \dyup(\mu)\text{ satisfies }\alpha_{\mu}(\Lambda Q)\leq \alpha\\\;0\text{ otherwise},\end{cases}\end{equation} it will suffice to demonstrate the following proposition:

\begin{prop}\label{flatcoef} There exist $\Lambda>0$, $\alpha>0$, $A>1$, and $\Delta>0$ such that for every measure $\mu$ and $Q\in \dyup(\mu)$ satisfying $\alpha_{\mu}(\Lambda Q)\leq \alpha$ and $\beta_{\mu}(Q)>0$, we have \begin{equation}\label{betasquarefn}\mathcal{S}_{\mu}^A(Q)\geq \Delta\beta_{\mu}(Q)^2 D_{\mu}(Q)^2\mathcal{I}_{\mu}(Q).\end{equation}
\end{prop}

Notice here that the $\beta$-number is present on the right hand side of (\ref{betasquarefn}).  It is not possible to estimate the square function coefficient in terms of the density alone (i.e., (\ref{betasquarefn}) couldn't possibly be true in general if one removes the $\beta_{\mu}(Q)^2$ term on the right hand side), as $\mu$ may well be the $s$-dimensional Hausdorff measure associated to an $s$-plane, in which case the left hand side of (\ref{betasquarefn}) equals to zero.

\section{The pruning lemma}


For an $n$-plane $L$ and $\beta>0$, $L_{\beta} = \{x\in \R^d: \dist(x, L)\leq \beta\}$ denotes the closed $\beta$-neighbourhood of $L$.

\begin{lem}\label{prunlem}  Let $R>0$.  Fix a measure $\mu$ with $\mu(B(0,R))>0$.  Suppose that for some hyperplane $H$ and $\beta>0$, we have
$$\frac{1}{\mu(B(0,R))}\int_{B(0,10R)}\Bigl(\frac{\dist(x,H)}{R}\Bigl)^2d\mu(x)\leq \beta^2.
$$
Then
\begin{equation}\begin{split}&\Bigl(\frac{\mu(B(0,R))}{R^s}\Bigl)^2\int_{B(0, 2R)\backslash H_{3\beta R}}\Bigl(\frac{\dist(x,H)}{R}\Bigl)^2d\mu(x)\\&\;\;\;\leq C\int_{B(0, 2R)}\int_{3R}^{4R}\Bigl|\int_{\R^d}\frac{x-y}{t^{s+1}}\varphi \Bigl(\frac{x-y}{t}\Bigl)d\mu(y)\Bigl|^2\frac{dt}{t}d\mu(x).
\end{split}\end{equation}
\end{lem}

\begin{proof}
We may assume that $R=1$ and $\mu(B(0,R))=1$.  Suppose that $H = b+ e^{\perp}$ for $b\in \R^d$ and $e\in \R^d$ with $|e|=1$, and for $x\in \R^d$ set $z_x = \langle x-b,e\rangle$.  Then
$$\Bigl|\int_{\R^d}(x-y)\varphi\Bigl(\frac{|x-y|}{t}\Bigl)d\mu(y)\Bigl|\geq \Bigl|\int_{\R^d}(z_x-z_y)\varphi\Bigl(\frac{|x-y|}{t}\Bigl)d\mu(y)\Bigl|.
$$
Fix $x\in B(0, 2)$ with $|z_x| = \dist(x,H)>3\beta$.  We will assume that $z_x>3\beta$.  Then
\begin{equation}\begin{split}\label{splittx}\int_{\R^d}\!\!(z_x-z_y)\varphi\Bigl(\frac{|x-y|}{t}\Bigl)d\mu(y)&\geq\!\! \int_{\{z_y< 2\beta\}}\!\!\!(z_x-z_y)\varphi\Bigl(\frac{|x-y|}{t}\Bigl)d\mu(y)\\
&-\int_{\{z_y>z_x\}}\!\!\!(z_y-z_x)\varphi\Bigl(\frac{|x-y|}{t}\Bigl)d\mu(y).
\end{split}\end{equation}
Notice that $\mu(\{y\in B(0,1): z_y\geq 2\beta\}) \leq \frac{1}{4\beta^2}\int_{B(0,1)}z_y^2d\mu \leq \frac{1}{4}$.  Consequently, if $t\in (3,4)$, we get that the first integral appearing on the right hand side of (\ref{splittx}) is at least $\frac{z_x}{3}\mu(B(0,1)\cap \{z_y<2\beta\})\geq \frac{z_x}{4}.$  On the other hand, the second integral on the right hand side of (\ref{splittx}) is at most
$$\int_{B(0,10)\cap\{z_y>3\beta\}}|z_y|d\mu(y)\leq \frac{1}{3\beta}\int_{B(0,10)\cap\{z_y>3\beta\}}|z_y|^2d\mu(y)\leq\frac{\beta}{3}\leq \frac{z_x}{9}.
$$
Thus
$$\Bigl|\int_{\R^d}(x-y)\varphi\Bigl(\frac{|x-y|}{t}\Bigl)d\mu(y)\Bigl|\geq \frac{|z_x|}{8}.
$$
It is easy to see that the conclusion also holds when $z_x<-3\beta$.  Squaring this inequality and integrating it yields,
\begin{equation}\begin{split}\nonumber\int_{B(0,2)\backslash H_{3\beta}}&|z_x|^2d\mu(x)\\
& \leq C \int_{B(0,2)}\int_3^4\Bigl|\int_{\R^d}\frac{x-y}{t^{s+1}}\varphi\Bigl(\frac{|x-y|}{t}\Bigl)d\mu(y)\Bigl|^2\frac{dt}{t}d\mu(x),
\end{split}\end{equation}
as required.
\end{proof}

We shall use this lemma as an alternative:

\begin{cor}[The Pruning  Alternative]\label{prunalt}  Fix a measure $\mu$ satisfying $\mu(B(0,R))>0$.  Fix $\Delta>0$.  Suppose that, for some $s$-plane $L$, and $R>0$, $$\frac{1}{\mu(B(0,R))}\int_{B(0,10R)}\Bigl(\frac{\dist(x,L)}{R}\Bigl)^2d\mu(x)\leq \beta^2.$$  Then, we have that either
\begin{equation}\begin{split}\nonumber\int_{B(0, 2R)}\int_{3R}^{4R}\Bigl|\int_{\R^d}&\frac{x-y}{t^{s+1}}\varphi\Bigl(\frac{x-y}{t}\Bigl)d\mu(y)\Bigl|^2\frac{dt}{t}d\mu(x)\\&\geq \Delta \beta^2\Bigl(\frac{\mu(B(0,R))}{R^s}\Bigl)^2\mu(B(0,R)),
\end{split}\end{equation}
or
$$\int_{B(0,2R)\backslash L_{3\beta(d-s) R}}\Bigl(\frac{\dist(x,L)}{R}\Bigl)^2d\mu(x)\leq C\Delta\beta^2\mu(B(0,R)).
$$
\end{cor}

Suppose $L = b+\{v_{d-s+1},\dots, v_d\}^{\perp}$ for an orthonormal set of vectors $v_{d-s+1}, \dots, v_d$.  One derives the corollary by applying Lemma \ref{prunlem} to the collection of $d-s$ hyperplanes $H^{(1)} = b+\{v_{d-s+1}\}^{\perp},\,\dots\,,H^{(d-s)} = b+\{v_{d}\}^{\perp}$, whose intersection is $L$.  One merely needs to notice that, on the one hand, for each $j\in \{1,\dots,d-s\}$, we have $\dist(\,\cdot\, , H^{(j)}) \leq \dist(\,\cdot\, , L)$.  But on the other hand $\dist(\,\cdot\, , L)\leq \sum_{j}\dist(\,\cdot\, , H^{(j)})$, and so for each $x\not\in L_{3\beta (d-s)R}$ there is some $j$ such that $x\not\in H^{(j)}_{3\beta R}$ and moreover $\dist(x,L)\leq (d-s)\dist(x, H^{(j)})$.

\section{The cylinder blow-up argument: the conclusion of the proof of Proposition \ref{flatcoef}}

We shall work in the following parameter regime: Fix $\Lambda \gg 1$ to be chosen later, and let $\alpha\rightarrow 0$, $\Delta\to 0$, and $A\to\infty$.

Suppose that, for each $k\in \mathbb{N}$, there is a measure $\wt\mu_k$, a cube $Q_k\in \dyup(\wt\mu_k)$ such that $\alpha_{ \wt\mu_k}(\Lambda Q_k)\leq \frac{1}{k}$, $\beta_{\wt\mu_k}(Q_k)>0$, and
$$\mathcal{S}_{\wt\mu_k}^k(Q_k)\leq \frac{1}{k}\beta_{\wt\mu_k}(Q_k)^2D_{\wt\mu_k}(Q_k)^2\I_{\wt\mu_k}(Q_k).
$$
Proposition \ref{flatcoef} will follow if we deduce a contradiction for some sufficiently large $\Lambda>0$.

Consider the measure $\mu_k = \frac{\wt\mu_k(\mathcal{L}_k\cdot)}{\I_{\wt\mu_k}(Q_k)}$.  Then $D_{\mu_k}(Q_0)=\mathcal{I}_{\mu_k}(Q_0)=1$.  The preimage of  $\dy$ under $\mathcal{L}_k$ is some lattice $\dy^{(k)}$ containing $Q_0$.  Passing to a subsequence we may assume that the lattices $\dy^{(k)}$ stabilize in some lattice $\dy'$.  Also observe that
\begin{equation}\label{smallsquare2}
\mathcal{S}_{\mu_k}^k(Q_0)\leq \frac{1}{k}\beta_{\mu_k}(Q_0)^2.
\end{equation}

Inasmuch as $Q_0\in \dyupk(\mu_k)$ and $\beta_{\mu_k}(Q_0)>0$, we infer from (\ref{densbetadoub})  that for any $N\geq 1$
\begin{equation}\label{easygrowth}D_{\mu_k}(NQ_0)\leq CN^{s+2\eps}.
\end{equation}

Notice that, since $\alpha_{\mu_k}(\Lambda Q_0)\leq \frac{1}{k}$, we have 
$\beta_{\mu_k}(\Lambda Q_0)\leq \tfrac{C}{\sqrt{k}}$.  From this and (\ref{easygrowth}) we find that $\beta_{\mu_k}(Q_0)\to 0$ as $k\to\infty$.

\subsection{Good density bounds for medium sized cubes containing $Q_0$}  In this section we shall prove the following result.

\begin{lem}\label{densityconstantlem}  There exists $C>0$ such that  if $Q'\in \dy^{(k)}$  with $\frac{\Lambda}{2}B_{Q_0}\supset \frac{1}{2}B_{Q'}\supset B_{Q_0}$, then
\begin{equation}\label{densityconstant}
\frac{1}{C}\leq D_{\mu_k}(Q')\leq C.
\end{equation}
(Here $C$ depends only on $d$ and $s$.)
\end{lem}

\begin{proof}The growth property (\ref{easygrowth}) ensures that $\mu_k(B_{\Lambda Q_0})\leq C(\Lambda)$.  Consequently, from the fact that $\alpha_{\mu_k}(\Lambda Q_0)\leq \tfrac{1}{k}$, we infer that for each $k$ there exists an $s$-plane $V_k$ that intersects $\tfrac{1}{4}B_{\Lambda Q_0}$ such that for every $f\in \Lip_0(B_{\Lambda Q_0})$  with $\|f\|_{\Lip}\leq 1$,
\begin{equation}\label{smallalpha}\Bigl|\int_{\R^d}\varphi_{\Lambda Q_0}fd[\mu_k-\vartheta_k\mathcal{H}^s|_{V_k}]\Bigl|\leq \frac{C(\Lambda)}{k},
\end{equation}
where $\vartheta_k = \frac{\I_{\mu_k}(\Lambda Q_0)}{\I_{\mathcal{H}^s|_{V_k}}(\Lambda Q_0)}$.

Since $D_{\mu_k}(Q_0)=1$, we readily see by testing (\ref{smallalpha}) with $f=\varphi_{Q_0}$ that $\vartheta_k\I_{\mathcal{H}^s|_{V_k}}(Q_0)\geq \frac{1}{2}$ if $k$ is large enough.  Thus the plane $V_k$ intersects $B_{Q_0}$.   Consequently, $\frac{1}{C}\ell(Q')^s\leq\I_{\mathcal{H}^s|_{V_k}}(Q')\leq C\ell(Q')^s$ whenever $\frac{\Lambda}{2}B_{Q_0}\supset \frac{1}{2}B_{Q'}\supset B_{Q_0}.$  But also $1\leq \I_{\mu_k}(3Q_0)\leq C$ from (\ref{easygrowth}).  Testing (\ref{smallalpha}) with $f=\varphi_{3Q_0}$ therefore yields that $\frac{1}{C}\leq \vartheta_k\leq C$ (for large $k$).     Finally, testing (\ref{smallalpha}) against $f=\varphi_{Q'}$, with $Q'$ as in the statement of the lemma, we infer that (\ref{densityconstant}) holds.\end{proof}

 Fix  $R$ to be an integer power of $2$ that satisfies $1\ll R\ll \Lambda$.  We choose a dyadic ancestor of $Q_0$ in $\dy'$, say $\wh{Q}_0$, of sidelength $16 R$.  Since the lattices $\dy^{(k)}$ stabilize, $\wh{Q}_0$ is a dyadic ancestor of $Q_0$ in the lattice $\dy^{(k)}$ for large enough $k$.  Insofar as $Q_0\in \dyupk(\mu_k)$, from (\ref{uprestate}) and (\ref{densityconstant}) we derive that
\begin{equation}\label{epsgrowth}\beta_{\mu_k}(\wh{Q}_0)\leq CR^{\eps}\beta_{\mu_k}(Q_0).
\end{equation}
Set $\beta_k = \beta_{\mu_k}(\wh Q_0)$.  Note that $\lim_{k\to \infty}\beta_k=0$.

\subsection{Concentration around the optimal least squares plane}

Denote by $L_k$ an optimal $s$-plane for $\beta_{\mu_k}(\wh Q_0)$.  Since $\I_{\mu_k}(Q_0)=1$, it is easily seen from Chebyshev's inequality that $L_k$ passes through $B_{Q_0} = B(0, 4\sqrt{d})$ for all sufficiently large $k$, so the closest point $x_k$ in $L_k$ to $0$ lies in $B_{Q_0}$.  Then clearly we have that
\begin{equation}\label{xktriv}B(x_k, \tfrac{r}{2})\subset B(0, r)\subset B(x_k, 2r) \text{ for every } r\geq 8\sqrt{d}.
\end{equation}

In this section our aim is to demonstrate the following lemma:

\begin{lem}\label{concentration}
There is a constant $C_1>0$, depending on $d$ and $s$, such that if $\wt\beta_k = C_1 \beta_k$, then
\begin{equation}\label{afterprune}\int_{B(0, 2R)\backslash L_{k,\wt\beta_k R}}\dist(x, L_k)^2d\mu_k(x)\leq \frac{C(R)\wt\beta_k^2}{ k},
\end{equation}
where $L_{k, \wt\beta_k R} = \{x\in \R^d: \, \dist(x, L_k)\leq \wt\beta_k R\}$.
\end{lem}
This is a much stronger concentration property around the plane $L_k$ than the one that the $\beta$-number alone provides us with.  It will play a crucial role in the subsequent argument.

\begin{proof}[Proof of Lemma \ref{concentration}] We look to apply the pruning alternative.  
Observe that, provided $k$ is large enough
\begin{equation}\begin{split}\label{smallconstinproof} \int_{B(0, 2R)}\int_{3R}^{4R}\Bigl|\int_{\R^d}&\frac{x-y}{t^{s+1}}\varphi\Bigl(\frac{|x-y|}{t}\Bigl)d\mu_k(y)\Bigl|^2\frac{dt}{t}d\mu_k(x)\\&\leq \mathcal{S}_{\mu_k}^k(Q_0)\stackrel{(\ref{smallsquare2})}{\leq} \frac{1}{k}\beta_{\mu_k}(Q_0)^2 \stackrel{(\ref{densityconstant})}{\leq}\frac{C R^s}{k}\beta_k^2.\end{split}\end{equation}
 On the other hand, using (\ref{densityconstant}) once again we derive that $\I_{\mu_k}(\wh{Q}_0)\leq C\mu_k(B(0,R))$, while $\varphi_{\wh{Q}_0}\geq 1$ on $B(0, 10R)$, 
 so we certainly have that
\begin{equation}\label{prunhypsat}\frac{1}{\mu_k(B(0,R))}\int_{B(0, 10R)}\frac{\dist(x, L_k)^2}{R^2}d\mu_k(x)\leq C\beta_k^2.
\end{equation}
Consider the alternative in Corollary \ref{prunalt}, with $\Delta = \frac{CR^s}{k}$, and $\beta = \wt\beta_k = C_1\beta_k$.  If $C_1$ is chosen appropriately in terms of $d$ and $s$, the inequality (\ref{smallconstinproof}) forces us into the first case of Corollary \ref{prunalt}, which is to say that
$$\int_{B(0, 2R)\backslash L_{k,\wt\beta_k R}}\frac{\dist(x, L_k)^2}{R^2}d\mu_k(x)\leq \frac{CR^s\wt\beta_k^2}{ k}\mu(B(0,R)), $$
as required.\end{proof}

\subsection{Stretching the measure around the least squares plane}  Let $\mathcal{A}^{(k)}$ denote a rigid motion that maps the $s$-plane $\{0\}\times\R^s$ (with $0\in \R^{d-s}$) to $L_k$ and $0\in \R^d$ to $x_k$. We introduce the co-ordinates $x=(x',x'')$, $x'\in \R^{d-s}$, $x''\in \R^{s}$.  Then from (\ref{xktriv}) and (\ref{afterprune}) we have
\begin{equation}\label{notmuchstrip}\int_{B(0,R)\backslash( \{0\}\times\R^s)_{\wt\beta_k R}}\frac{|x'|^2}{\wt\beta_k^2}d(\mu_k\circ\mathcal{A}^{(k)})(x',x'')\leq \frac{C(R)}{k}.
\end{equation}
We define the squash mapping $\mathcal{S}_{\beta}(x) = (\beta x', x'')$ for $\beta>0$, along with the stretched measure  $$\nu_k(\,\cdot\,) = \mu_k(\mathcal{A}^{(k)}\circ \mathcal{S}_{\wt\beta_k}(\,\cdot\,)).$$

Since $\wt\beta_k<1$ for large enough $k$, we have $\nu_k(B(0,N))\leq (\mu_k\circ \mathcal{A}^{(k)})(B(0,N))$ for $N>0$.  As $\mu_k$ satisfies (\ref{easygrowth}), we see that we may pass to a subsequence such that $\nu_k$ converge weakly to a measure $\nu$.

For $m\in \mathbb{N}$, denote by $B^m(z,r)$ the $m$-dimensional ball centred at $z\in \R^m$ with radius $r>0$.  
Under our change of variables, the inequality (\ref{notmuchstrip}) becomes
\begin{equation}\label{nunotmuchstrip}\int_{[\mathcal{S}_{\wt\beta_k}^{-1}(B(0,R))]\backslash (\overline{B^{d-s}(0,R)}\times \R^s)}|x'|^2d\nu_k(x)\leq \frac{C(R)}{k}.
\end{equation}
Whence,
\begin{equation}\label{nothingoutstrip}\supp(\nu)\cap [\R^{d-s}\times B^s(0,R)]\subset \overline{B^{d-s}(0,R)}\times B^s(0,R).
\end{equation}
On the other hand, $\mu_k\circ\mathcal{A}^{(k)}(B(0, 8\sqrt{d}))\geq D_{\mu_k}(Q_0)=1$, and so  from (\ref{nunotmuchstrip}) we derive that $\nu_k(\overline{B^{d-s}(0,R)}\times B^s(0, 8\sqrt{d}))\geq 1-\frac{C(R)}{k}$.
Thus
$$\nu(\R^{d-s}\times \overline{B^s(0, 8\sqrt{d})}) = \nu(\overline{B^{d-s}(0,R)}\times \overline{B^s(0, 8\sqrt{d})})\geq 1.$$

\begin{lem}\label{nuweakconv}  The following three properties hold:

\begin{enumerate}
\item  If $f\in \Lip_0(B(0,R))$, then
$$\lim_{k\to\infty}\int_{\R^d}f(x',x'')d(\mu_k\circ \mathcal{A}^{(k)})(x',x'') = \int_{\R^d}f(0,x'')d\nu(x',x'').
$$
\item If $f\in \Lip_0(\R^{d-s}\times B^s(0,R))$, then
$$\lim_{k\to\infty}\int_{\R^d}f\bigl(\frac{x'}{\wt\beta_k},x''\bigl)d(\mu_k\circ \mathcal{A}^{(k)})(x',x'') = \int_{\R^d}f(x',x'')d\nu(x',x'').
$$
\item If $t\in (0, \tfrac{R}{8})$, then
\begin{equation}\begin{split}\nonumber\liminf_{k\to\infty}&\int_{B(0,R/2)}\Bigl|\int_{\R^d}\frac{x'-y'}{\wt\beta_k}\varphi\Bigl(\frac{|x-y|}{t}\Bigl)d(\mu_k\circ\mathcal{A}^{(k)})(y',y'')\Bigl|^2d(\mu_k\circ\mathcal{A}^{(k)})(x',x'')\\
&\geq  \int_{\R^{d-s}\times B^s(0,R/2)}\Bigl|\int_{\R^d}(x'-y')\varphi\Bigl(\frac{|x''-y''|}{t}\Bigl)d\nu(y',y'')\Bigl|^2d\nu(x',x'').
\end{split}\end{equation}
\end{enumerate}
\end{lem}

The proof is a slightly cumbersome exercise in weak convergence, using the property (\ref{nunotmuchstrip}).  As such, we postpone the proof Section \ref{nuweaksec}.

\subsection{The limit measure $\nu$ is a cylindrically $\varphi$-symmetric measure.} For $r\in (0, \tfrac{R}{8})$, let us examine the inequality
$$\int_{B(0, R)}\Bigl|\int_{\R^d}(x-y)\varphi\Bigl(\frac{|x-y|}{r}\Bigl)d\mu_k(y)\Bigl|^2d\mu_k(x)\leq \frac{C(R)\wt\beta_k^2}{k}
$$
(see (\ref{smallconstinproof})).  We would like to see what happens to this inequality under the change of variables that takes $\mu_k$ to $\nu_k$.  First notice that, because of (\ref{xktriv}) (and the fact that a rigid motion is an isometry)
\begin{equation}\begin{split}\nonumber\int_{B(0,R/2)}\Bigl|\int_{\R^d} \frac{x'-y'}{\wt\beta_k}\varphi\Bigl(\frac{|x-y|}{r}\Bigl)&d(\mu_k\circ\mathcal{A}^{(k)})(y',y'')\Bigl|^2d(\mu_k\circ \mathcal{A}^{(k)})(x',x'')\\&\leq \frac{C(R)}{k}.
\end{split}\end{equation}
From this, we deduce  from   item (3) of Lemma \ref{nuweakconv} that
\begin{equation}\begin{split}\label{verticalvampire}\int_{\R^d}&(x'-y')\varphi\Bigl(\frac{|x''-y''|}{r}\Bigl)d\nu(y',y'')=0\\ &\text{ for every }(x',x'')\in \supp(\nu)\cap [\R^{d-s}\times B^s(0, R/2)].
\end{split}\end{equation}

We will establish the following lemma:

\begin{lem}\label{toosmall}  There exists a constant $C>0$ such that for sufficiently large $k$,
$$\beta_{\mu_k}(Q_0)\leq \frac{C}{R}\wt\beta_k.
$$
\end{lem}

The estimate in this lemma is inconsistent with (\ref{epsgrowth}) if $R$ is large enough.  A contradictory choice of $R$ is possible once $\Lambda$ is chosen large enough in terms of $d$ and $s$.  As such, we will have completed the proof of Proposition \ref{flatcoef} once the lemma is established.

 The key to proving Lemma \ref{toosmall} will be to show that, when restricted to $\R^{d-s}\times B^s(0, R/2)$, \emph{the support of $\nu$ is the graph of an $\R^{d-s}$-valued harmonic function on $B^s(0,R/2)$}.  For this, we shall use the fact that $\alpha_{\mu_k}(\Lambda Q_0)$ tends to zero as $k\to\infty$ in a more substantial way than we have up to this point.

\subsection{Large projections of the limit measure} In this section we shall prove the following result.

\begin{lem}\label{bigproj} There exists $\vartheta_0>0$ such that for every $f:\R^s\to \R$, $f\in \Lip_0(B^s(0, \tfrac{3R}{4}))$, we have
$$\int_{\R^{d}}f(x'')d(\nu - \vartheta_0\mathcal{H}^s|_{\{0\}\times \R^s})(x',x'')=0
$$
\end{lem}

\begin{proof} Recall (see the proof of Lemma \ref{densityconstantlem}) that for every $k$ there is an $s$-plane $V_k$ for which (\ref{smallalpha}) holds for every $f\in \Lip_0(B_{\Lambda Q_0})$ with $\|f\|_{\Lip}\leq 1$, and $\frac{1}{C}\leq\vartheta_k\leq C$.  Also recall that $L_k$ is an optimal $s$-plane for $\beta_{\mu_k}(\wh Q_0)$.  Both $V_k$ and $L_k$ pass through $B_{Q_0}$ if $k$ is sufficiently large.

Consider a cut-off function $h\in \Lip_0(B(0, R))$, with $h\equiv 1$ on $B(0,3R/4)$ and $\|h\|_{\Lip}\leq 1$.  Then the function $x\mapsto h(x)(\dist(x,L_k))^2$ is $C(R)$-Lipschitz, and so, by (\ref{smallalpha}) and the definition of the $\beta$-coefficient, we infer that
$$\int_{B(0, 3R/4)}\dist(x,L_k)^2d\mathcal{H}^k|_{V_k}(x)\leq \frac{C(R)}{k}+C(R)\wt\beta_k^2.$$
Given that the planes $L_k$ and $V_k$ both pass through $B_{Q_0}$, this implies that the intersection of the plane $[\mathcal{A}^{(k)}]^{-1}(V_k)$ with the ball $B(0, \tfrac{3R}{4})$ lies within a $C(R)\omega_k$ neighbourhood of $[\{0\}\times \R^s] \cap B(0, \tfrac{3R}{4})$, where $\omega_k \to 0$ as $k\to\infty$. Consequently, if $F\in \Lip_0(B(0, \tfrac{3R}{4}))$, $\|F\|_{\Lip}\leq 1$, then
\begin{equation}\label{alphacloseplane}\Bigl|\int_{\R^d} F(x',x'')d(\mu_k\circ\mathcal{A}^{(k)} - \vartheta_k\mathcal{H}^s|_{\{0\}\times \R^s})(x',x'')\Bigl|\leq C(R)\omega_k.
\end{equation}

Passing to a subsequence so that $\vartheta_k$ converges to $\vartheta_0$, we get from item (1) of Lemma \ref{nuweakconv} that
$$\int_{\R^d} F(0,x'')d(\nu - \vartheta_0\mathcal{H}^s|_{\{0\}\times \R^s})(x',x'')=0.
$$
The lemma follows immediately from this statement.\end{proof}

As a consequence of the lemma, note that whenever $x''\in B^s(0, \tfrac{R}{2})$ and $t<R/8$, we have
\begin{equation}\begin{split}\nonumber \int_{\R^{d-s}\times B^s(x'',t)}\varphi\Bigl(\frac{|x''-y''|}{t}\Bigl)d\nu(y',y'') &= \vartheta_0\I_{\mathcal{H}^s}(B^s(x'',t))\\&=\vartheta_0\I_{\mathcal{H}^s}(B^s(0,t)).
\end{split}\end{equation}

\subsection{The final contradiction: The proof of Lemma \ref{toosmall}}

From the observations of the previous section along with the property (\ref{verticalvampire}), we find if $(x',x'')\in \supp(\nu)\cap [\R^{d-s}\times B^s(0,\tfrac{R}{2})]$ and $r\in (0, R/8)$, then
$$x' = \frac{1}{\vartheta_0\I_{\mathcal{H}^s}(B^s(0,r))}\int_{\R^{d}}y'\varphi\Bigl(\frac{|x''-y''|}{r}\Bigl)d\nu(y',y'').
$$
This formula determines $x'$ in terms of $x''$.  From this, we derive that $\supp(\nu)\cap (\R^{d-s}\times B^s(0, \tfrac{R}{2}))$ is a graph given by $\{(u(x''),x''): x''\in B^s(0, \tfrac{R}{2})\}$ for some $u: B^s(0, \tfrac{R}{2})\to \overline{B^{d-s}(0,R)}$.  As, for each Borel set $E\subset B^s(0, \tfrac{R}{2})$,
\begin{equation}\label{cylindermeasure}\nu(\R^{d-s}\times E) = \nu(\overline{B^{d-s}(0,R)}\times E) = \vartheta_0\mathcal{H}^s(E),
\end{equation}
we have that whenever $B^s(x'', 2r)\subset B^s(0, R/2)$,
$$u(x') =  \frac{1}{\I_{\mathcal{H}^s}(B(x'',r))}\int_{\R^s}u(y'')\varphi\Bigl(\frac{|x''-y''|}{r}\Bigl)d\mathcal{H}^s(y'').
$$
This certainly ensures that $u$ is a smooth function, but moreover it is harmonic.  Indeed, for each $x''\in B(0, R/2)$ we have that for small enough $r$,
\begin{equation}\begin{split}\label{polar}0 &= \int_{\R^s}\varphi\Bigl(\frac{|x''-y''|}{r}\Bigl)[u(y'')-u(x'')]d\mathcal{H}^s(y'') \\ &=c\int_0^{2r}\int_{\mathbb{S}^{s-1}}[u(x''+t\omega) -u(x'')]d\sigma(\omega)\varphi\Bigl(\frac{t}{r}\Bigl)t^s\frac{dt}{t},
\end{split}\end{equation}
where $d\sigma$ denotes the surface area measure on the unit $s$-sphere $\mathbb{S}^{s-1}$.  With $\Delta_s$ denoting the Laplacian in $\R^s$, we infer from Taylor's formula (or the divergence theorem) that
$$\int_{\mathbb{S}^{s-1}}[u(x''+t\omega) -u(x'')]d\sigma(\omega) = ct^2\Delta_s u(x'') + O(t^3) \text{ as }t\to 0
$$
for some constant $c>0$.  
Plugging the preceding identity into (\ref{polar}) yields that $r^{s+2}|\Delta u(x'')|\leq Cr^{s+3}$ for all small $r$.  Hence $\Delta u(x'')=0$ for  $x''\in B(0, \tfrac{R}{2})$.


Since $|u(x'')|\leq R$ for every $x''\in B^s(0, \frac{R}{2})$ (see (\ref{nothingoutstrip})), standard gradient estimates yield that $|\nabla u(x'')|\leq C$ if $x''\in B^s(0, \tfrac{R}{4})$.  In order to prove Lemma \ref{toosmall}, we shall employ the following simple estimate for harmonic functions.  We introduce the notation $\dashint_E fd\mathcal{H}^s := \frac{1}{\mathcal{H}^s(E)}\int_E fd\mathcal{H}^s$.

\begin{lem}\label{harmpoly}  If $B^s(x'',r)\subset B(0, \tfrac{R}{16})$, then
\begin{equation}\begin{split}\nonumber\dashint_{B^s(x'',r)} |u(y'') - & u(x'') - D u(x'')(y''-x'')|^2d\mathcal{H}^s(y'') \\&\leq C\Bigl(\frac{r}{R}\Bigl)^4\dashint_{B^s(0, \tfrac{R}{2})}|u|^2d\mathcal{H}^s.
\end{split}\end{equation}
\end{lem}

\begin{proof} Note that if $y''\in B^s(x'',r)$, then Taylor's theorem ensures that for some $z''\in B^s(x'',r)$,
$$|u(y'') -  u(x'') - D u(x'')(y''-x'')|\leq Cr^2|D^2u(z'')|.
$$
But now since $u$ is harmonic, from standard gradient estimates and the mean value property we obtain that
$$|D^2u(z'')|\leq \frac{C}{R^2}\sup_{B(z'', \tfrac{R}{4})}|u|\leq \frac{C}{R^2}\dashint_{B^s(x'', R/2)}|u|d\mathcal{H}^s.
$$
Squaring both sides of the resulting inequality, taking the integral average over $B(x'', r)$, and using the Cauchy-Schwartz inequality, we arrive at the desired statement.\end{proof}

Written in terms of $\nu$, the previous estimate, along with the property (\ref{cylindermeasure}), ensure that there exist a $(d-s)\times s$ matrix $A$ and a vector $b\in \R^s$ such that
\begin{equation}\begin{split}\label{betterplaneapprox} \dashint_{\R^{d-s}\times B^s(0,300\sqrt{d}\ell(Q_0))}&|x'-Ax''-b|^2d\nu(x',x'') \\&\leq \Bigl(\frac{C}{R}\Bigl)^2\dashint_{\R^{d-s}\times B^s(0, \tfrac{R}{2})}\Bigl(\frac{|x'|}{R}\Bigl)^2d\nu(x',x'').
\end{split}\end{equation}
Furthermore we have $A=\nabla u(0)$, and $b = u(0)$, and so  $|b|\leq R$ and $|A|\leq C$.

Consider the function $f:\R^{s}\to\R$ given by $f(x'')=\varphi\bigl(\frac{|x''|}{100\sqrt{d}}\bigl)$ and fix a non-negative function $g\in \Lip_0(B^{d-s}(0, 2R)$ with $g\equiv 1$ on $\overline{B^{d-s}(0,R)}$.  Then from statement (2) of Lemma \ref{nuweakconv} we get that
\begin{equation}\begin{split}\nonumber&\int\limits_{\R^d}g(x')f(x'')|x'-Ax''-b|^2d\nu(x',x'')\\
& = \lim_{k\to\infty}\frac{1}{\wt\beta_k^2}\int\limits_{\R^d}g(\tfrac{x'}{\wt\beta_k})f(x'')|x'-\wt\beta_k A x''-\wt\beta_k b|^2d(\mu_k\circ \mathcal{A}^{(k)})(x',x'')\\
&\geq \limsup_{k\rightarrow \infty}\frac{1}{\wt\beta_k^2}\int_{\{|x'|\leq \wt\beta_k R\}}\varphi_{25 Q_0}(x)|x'-\wt\beta_k A x''-\wt\beta_k b|^2d(\mu_k\circ \mathcal{A}^{(k)})(x',x'').
\end{split}\end{equation}
(In the final line we have used the trivial  observation that $f(x'')\geq \varphi_{25Q_0}(x)$ for $x=(x',x'')\in \R^d$.)  On the other hand, using (\ref{notmuchstrip}) and (\ref{densityconstant}), statement (2) of Lemma \ref{nuweakconv} ensures that
\begin{equation}\begin{split}\nonumber&\dashint_{\R^{d-s}\times B^s(0, \tfrac{R}{2})}\Bigl(\frac{|x'|}{R}\Bigl)^2d\nu(x',x'')\\
&\leq \liminf_{k\rightarrow \infty}\frac{C}{\I_{\mu_k}(\wh{Q}_0)}\int_{B(0, \frac{R}{2})}\Bigl(\frac{|x'|}{R\wt\beta_k}\Bigl)^2d(\mu_k\circ\mathcal{A}^{(k)})(x',x'')\\
&\leq \liminf_{k\rightarrow \infty}\frac{C\beta_k^2}{\wt\beta_k^2}\leq C.
\end{split}\end{equation}
Comparing the previous two observations with (\ref{betterplaneapprox}), and using our bounds for $A$ and $b$, we find for all sufficiently large $k$ some $s$-plane $\wt{L}_k$ with $B(0,\tfrac{R}{2})\cap \wt L_k\subset \{\dist(x,L_k)\leq C\wt\beta_k R\}$, such that
$$\frac{1}{\wt\beta_k^2}\int_{\{\dist(x,L_k)\leq \wt\beta_kR\}}\varphi_{25Q_0}(\mathcal{A}^{(k)}x)\dist(x, \wt L_k)^2d\mu_k(x)\leq \frac{C}{R^2}.
$$
On the other hand, if $x\in B(0, \tfrac{R}{2})$ satisfies $\dist(x,L_k)> \wt\beta_k R$, then certainly  $\dist(x, \wt L_k)\leq C\dist(x, L_k)$.  Whence, from (\ref{afterprune}), we infer that all for large enough $k$,
\begin{equation}\begin{split}\nonumber&\int_{\{\dist(x,L_k)> \wt\beta_k R\}}\varphi_{25Q_0}(\mathcal{A}^{(k)}x)\dist(x, \wt L_k)^2d\mu_k(x)\\&\leq C\int_{B(0, \tfrac{R}{2})\cap\{\dist(x,L_k)> \wt\beta_k R\}}\dist(x, L_k)^2d\mu_k(x)\leq \frac{C(R)}{k}\wt\beta_k^2\leq \frac{1}{R^2}\wt\beta_k^2.
\end{split}\end{equation}
Notice that (\ref{xktriv}) ensures that $\varphi_{25Q_0}(\mathcal{A}^{(k)}\,\cdot\,)\geq \varphi_{Q_0}$.  Consequently, by combining our observations,  we see that for sufficiently large $k$,
\begin{equation}\label{betatoosmall2}\beta_{\mu_k}(Q_0)\leq \frac{C}{R}\wt\beta_{k},
\end{equation}
and so Lemma \ref{toosmall} is proved.\\

\subsection{The proof of Lemma \ref{nuweakconv}}\label{nuweaksec} We now turn to proving Lemma \ref{nuweakconv}.

\begin{proof}[Proof of Lemma \ref{nuweakconv}] Note the identity
$$\int_{\R^d}f(x',x'')d(\mu_k\circ\mathcal{A}^{(k)})(x',x'') = \int_{\R^d}f(\wt\beta_k x',x'')d\nu_k(x',x'').
$$
By replacing $f$ in this identity with $(x',x'')\mapsto f(\tfrac{x'}{\wt\beta_k}, x'')$, we see that item (2) of the Lemma follows directly from the weak convergence of $\nu_k$ to $\nu$.   Fix $g\in \Lip_0(B^{d-s}(0, 2R))$  satisfying $g\equiv 1$ on $\overline{B^{d-s}(0,R)}$.  Because of (\ref{nunotmuchstrip}), if $f\in \Lip_0(B(0,R))$, $\|f\|_{\Lip}\leq 1$, then
$$\Bigl|\int_{\R^d}f(\wt\beta_k x',x'')d\nu_k(x',x'') -  \int_{\R^d}g(x',x'')f(\wt\beta_k x',x'')d\nu_k(x',x'')\Bigl|\leq \frac{C(R)}{k}.$$
But the function $(x',x'')\mapsto g(x',x'')f(\wt\beta_k x',x'')$ converges to the function $(x',x'')\mapsto g(x',x'')f(0,x'')$ uniformly on $\overline{B^{d-s}(0,2R)}\times B^s(0,R)$, and
$\int_{\R^d}f(0,x'')d\nu(x',x'') = \int_{\R^d}g(x',x'')f(0,x'')d\nu(x',x'')$. Item (1) is follows immediately from these two observations.

To prove item (3), we shall look to apply Lemma \ref{stupidlemma}.  For $t\in (0, \tfrac{R}{8})$, consider the integral $I_k$ given by
$$\int\limits_{\mathcal{S}_{\wt\beta_k}^{-1}(B(0,\tfrac{R}{2}))}\!\!\!\!\Bigl|\int_{\R^d}(x'-y')\varphi\Bigl(\frac{|(\wt\beta_k [x'-y'],[x''-y''])|}{t}\Bigl)d\nu_k(y',y'')\Bigl|^2d\nu_k(x',x'').
$$
Notice that if we choose $f\in \Lip_0(B^s(0,R))$ with $f \equiv 1$ on $B^s(0, \tfrac{3R}{4})$, then inserting a factor of $f(y'')f(x'')$ in the inner integral does not affect the value of the double integral.  Consider the measure $d\wt\nu_k(x',x'') = f(x'')d\nu_k(x',x'')$.  The error introduced by replacing $I_k$ with the integral $\wt{I}_k$, defined by
\begin{equation}\begin{split}\label{horribleexpression}\int_{\mathcal{S}_{\wt\beta_k}^{-1}(B(0,\tfrac{R}{2}))}\Bigl|\int_{\R^d}\Bigl[(x'-y')&\varphi\Bigl(\frac{|(\wt\beta_k [x'-y'],[x''-y''])|}{t}\Bigl)\\&\cdot g(x')g(y')\Bigl]d\wt\nu_k(y',y'')\Bigl|^2d\wt\nu_k(x',x''),
\end{split}\end{equation}
is bounded by a constant multiple of
\begin{equation}\begin{split}\nonumber&\int_{[\mathcal{S}_{\wt\beta_k}^{-1}(B(0,\tfrac{R}{2}))]\backslash [\overline{B^{d-s}(0,R)}\times B^s(0,\tfrac{R}{2})]}|x'|^2d\nu_k(x',x'')\nu_k(\mathcal{S}_{\wt\beta_k}^{-1}(B(0,R)))^2\\
&+ \nu_k(\mathcal{S}_{\wt\beta_k}^{-1}(B(0,\tfrac{R}{2})))\Bigl(\int_{[\mathcal{S}_{\wt\beta_k}^{-1}(B(0,R))]\backslash [\overline{B^{d-s}(0,R)}\times B^s(0,R)]}|y'|d\nu_k(y',y'')\Bigl)^2.
\end{split}\end{equation}
From (\ref{nunotmuchstrip}) we therefore infer that $\lim_{k\to\infty}|I_k - \wt I_k|\leq \lim_{k\to\infty}\frac{C(R)}{k}=0$. (Note that, from (\ref{xktriv}), $\nu_k(\mathcal{S}_{\wt\beta_k}^{-1}(B(0,\tfrac{R}{2})))\leq \mu_k(B(0, R))\leq CR^s.$)

Observe that the function $$\psi_k(x,y) = (x'-y')\varphi\Bigl(\frac{|\wt\beta_k (x'-y'),x''-y''|}{t}\Bigl)g(x')g(y')$$ converges uniformly as $k\to\infty$ to $$\psi(x,y) = (x'-y')\varphi\Bigl(\frac{|x''-y''|}{t}\Bigl)g(x')g(y'),$$ and for each $x\in \R^d$, $\supp(\psi_k(x,\cdot))\subset B(x, 2\sqrt{d}R)$.  Clearly $\sup_k\|\psi_k\|_{\Lip}<\infty$, as the $\wt\beta_k$ factor can only decrease the Lipschitz norm of $\varphi$.  Appealing to Lemma \ref{stupidlemma} with the sequence of measures $\wt\nu_k$, which converge weakly to the measure  $d\wt\nu(x',x'') = f(x'')d\nu(x',x'')$, and $U = B^{d-s}(0,2R)\times B^s(0, \tfrac{R}{2})$, we infer that $\liminf_k I_k$ is at least
$$\int\limits_{B^{d-s}(0,2R)\times B^s(0, \frac{R}{2})}\Bigl|\int_{\R^d}(x'-y')\varphi\Bigl(\frac{|x''-y''|}{t}\Bigl)g(x')g(y')d\wt\nu(y',y'')\Bigl|^2d\wt\nu(x',x''),
$$
and, after recalling the basic properties of $g$ and $f$, this proves (3).
\end{proof}



 \end{document}